\documentclass[11pt]{amsart}
\usepackage[utf8]{inputenc}

\usepackage{amsmath,amssymb,amsfonts,amsthm}
\usepackage{graphicx}
\usepackage[colorlinks=true]{hyperref}
\usepackage{epsfig, enumerate}
\usepackage{tikz}
\usetikzlibrary{arrows.meta,decorations.pathreplacing,positioning}

\numberwithin{equation}{section}
\newcommand{\R}{\mathbb{R}}
\newcommand{\N}{\mathbb{N}}
\newcommand{\Z}{\mathbb{Z}}

\newcommand{\W}{\mathcal{W}}
\newcommand{\CZ}{\mathcal{CZ}}

\newcommand{\beqnn}{\begin{eqnarray*}}
\newcommand{\eeqnn}{\end{eqnarray*}}
\newcommand{\beqn}{\begin{eqnarray}}
\newcommand{\eeqn}{\end{eqnarray}}
\newcommand{\beq}{\begin{equation}}
\newcommand{\eeq}{\end{equation}}

\theoremstyle{plain}
\newtheorem{thm}{Theorem}[section]
\newtheorem{prop}[thm]{Proposition}

\newtheorem{lem}[thm]{Lemma}
\newtheorem{cor}[thm]{Corollary}
\newtheorem{rmk}[thm]{Remark}
\newtheorem{defi}{Definition}[section]

\begin{document}

\title{Local centralizer rigidity for twisted Weyl chamber flows}
\author{Zhijing Wendy Wang}
\address{Department of Mathematics, University of Chicago, Chicago, IL 60637, USA}
\email{zhijingw@uchicago.edu}

\begin{abstract}
 We prove local centralizer rigidity results for elements of Weyl chamber
flows and twisted Weyl chamber flows on compact homogeneous spaces.  For
generic elements in the twisted setting, sufficiently large dimension of the centralizer
forces smooth conjugacy to an algebraic model.  For many non-generic
elements, we prove analogous rigidity under a virtual centralizer-isomorphism
hypothesis.
\end{abstract}

\maketitle

\tableofcontents

\section{Introduction}

A basic problem in smooth dynamics is to understand the symmetries of a
diffeomorphism.  For a diffeomorphism \(f\) of a compact manifold \(X\),
these symmetries are encoded by the smooth centralizer
\[
\mathcal Z(f)=\{g\in\mathrm{Diff}^\infty(X): gf=fg\}.
\]
Smale conjectured that a generic diffeomorphism has trivial centralizer,
i.e. only its iterates commute with it \cite{smale}.  Thus a large
centralizer is a non-generic phenomenon, and one expects it to reflect some
underlying algebraic structure.

The centralizer-rigidity problem asks whether this expectation is stable
under perturbation: if an algebraic diffeomorphism has a large centralizer,
and a small perturbation still has a centralizer of comparable size or
structure, must the perturbation be smoothly conjugate to an algebraic
model?  This question is especially natural for partially hyperbolic
homogeneous systems, where the algebraic model already contains higher-rank
hyperbolic structure.

By an algebraic model, we mean an affine diffeomorphism $f_0 = L_a \circ \Psi$ on a homogeneous space $X=G/\Gamma$, which is a composition of a left translation $L_a$ and an automorphism $\Psi$ of \(G\) preserving \(\Gamma\), where $G$ is a Lie group, $\Gamma<G$ is a cocompact lattice. In a recent paper, Damjanović, Wilkinson, Wu, and Xu \cite{DWWX} proved that if $f_0$ is not a $K$-system, then there exist $C^1$-perturbations of $f_0$, whose centralizer is ``extremely large" (contains a copy of $C^\infty_c(0,1)$ or $\mathrm{Diff}^\infty_c(0,1)$). So the natural framework is to only consider an affine $K$-system $f_0$. In the same paper, they also presented a series of conjectures. Namely, if $f_0$ is an affine $K$-system, and if $f\in \mathrm{Diff}^\infty(X)$ is a sufficiently small $C^1$-perturbation of $f_0$, they conjectured \begin{enumerate}
    \item the centralizer of $f$ is a Lie group;
    
 \item if $\mathcal Z^\infty(f_0)$ has no rank-$1$ factor, then either $\mathcal Z^\infty(f)$ has a rank-$1$ factor, or $f$ is $C^\infty$ conjugate to an affine $K$-system on $X$; 
 
     \item if $\mathcal Z^\infty(f_0)$ has no rank-$1$ factor, and $\mathcal Z^\infty(f_0)\doteq \mathcal Z^\infty(f)$ as Lie groups, then $f$ is $C^\infty$ conjugate to an affine $K$-system on $X$.    
\end{enumerate}

Here, we say a group action \emph{has a rank-$1$ factor}, if it has a smooth factor which is a compact extension of a $\Z$- or $\R$-action. 

Witte \cite[Proposition 4.22]{WitteZeroEntropy} showed that for a weakly mixing affine diffeomorphism, the radical of the Lie group is nilpotent. Together with the Levi decomposition, this suggests that the natural
ambient groups for affine K-systems (that are weakly mixing) are semidirect products
\(G_0\ltimes N\), where \(G_0\) is semisimple and \(N\) is nilpotent.

In this work, we consider elements of the Weyl chamber flow acting on a semisimple homogeneous space; and elements of the twisted Weyl chamber flow acting on the homogeneous space associated with the semidirect product $G=G_0\ltimes \R^N$. Theorem~\ref{czlieg} offers an affirmative answer to Conjecture 1 of \cite{DWWX} for non-trivial elements of the (twisted) Weyl chamber flow. Furthermore, Theorems~\ref{cenrigtwgen} and  \ref{cenrignongen:twisted} collectively address Conjecture 3 of \cite{DWWX} for various non-trivial, potentially non-generic, elements of the (twisted) Weyl chamber flow.

This provides a first example confirming these conjectures for some affine $K$-systems of semidirect products, extending the centralizer rigidity results beyond nilmanifolds or semisimple homogeneous spaces. 

There is also a second viewpoint.  In the algebraic model, the large
centralizer of a single element comes from the ambient higher-rank action.  Classical local rigidity
of higher-rank algebraic actions assumes that the whole action is perturbed.
In the setting of this paper, only one element is assumed to be close to the
algebraic model. Thus the results may also be viewed as a
semi-local form of Zimmer-type rigidity. 

\subsection{Setting and main results}

We consider the following settings (see Section \ref{semisimple} for more detail):

\begin{enumerate}
    \item \textbf{semisimple homogeneous space:} Let \(X = G_0 / \Gamma_0\), where \(G_0\) is a real-split semisimple Lie group with finite center, and each simple component has rank at least two. Let \(\Gamma_0\) be a cocompact lattice in \(G_0\). The Weyl chamber flow in \(X\) is given by left multiplication by a split Cartan subgroup of \(G_0\). Since $G_0$ is real-split, the adjoint action of the Weyl chamber flow induces a splitting of the Lie algebra of $G_0$. 
    
    We take \(f_0\) to be an element of this flow, defined by left multiplication by a non-trivial element of the split Cartan subgroup (meaning that in each simple component, the eigenvalues of \(\operatorname{Ad}(f_0)\) do not all have modulus one). 

    A basic example is \(G_0 = \mathrm{SL}_n(\mathbb{R})\), with \(f_0\) being a non-trivial diagonal matrix acting by left multiplication on \(G_0/\Gamma_0\). An explicit example is \(G_0 = \mathrm{SL}_4(\mathbb{R})\), where a Weyl chamber flow is the left multiplication of $\mathrm{diag}(e^{t_1}, e^{t_2}, e^{t_3}, e^{t_4}): \sum_{i=1}^4 t_i=0,$ with roots $t_i-t_j, i\neq j$. We may take any non-trivial element, e.g. \(f_0 = \mathrm{diag}(e^{2}, e^{2}, e^{-1}, e^{-3})\) acting on \(\mathrm{SL}_4(\mathbb{R})/\Gamma_0\).

    \item \textbf{Twisted semisimple homogeneous space:} Let \(X = (G_0 \ltimes_\rho \mathbb{R}^N) / (\Gamma_0 \ltimes_\rho \mathbb{Z}^N)\), where \(G_0\) is a real-split semisimple Lie group as in item 1, \(\rho: G_0 \to \mathrm{SL}_N(\mathbb{R})\) is a representation with no trivial component, and \(\Gamma_0\) is a cocompact lattice in \(G_0\) such that \(\Gamma_0=\rho^{-1}(\mathrm{SL}_N(\mathbb{Z})\cap \rho(G_0))\). The twisted Weyl chamber flow on \(X\) is the action induced by left multiplication by elements of \(A \ltimes \{0\}\) on \(G_0 \ltimes_\rho \mathbb{R}^N\), where \(A=\exp \mathfrak h\) is a split Cartan subgroup of \(G_0\). The adjoint action on the Lie algebra level induces a splitting in the Lie algebra $\mathfrak{g}(G_0\ltimes \R^N)=\mathfrak{g}(G_0)\times \R^N$. The corresponding eigenvalue functionals for the adjoint action of $\mathrm{ad}|_{\mathfrak h}$ on
\(\mathfrak g_0\) are the roots, while the eigenvalue functionals for the
derived representation \(d\rho|_{\mathfrak h}\) on \(\mathbb R^N\) are the weights. For example, if $G=\mathrm {SL}_4(\R)\ltimes_\rho \R^N$, where $\rho$ is a direct sum of standard representations, then a Weyl chamber flow is given by $\mathrm{diag}(e^{t_1}, e^{t_2}, e^{t_3}, e^{t_4})\ltimes \{0\}: \sum_{i=1}^4 t_i=0,$ the roots are $t_i-t_j, i\neq j$ and the weights are $t_i, 1\le i\le 4$.
    
    We take the affine diffeomorphism \(f_0\) to be left multiplication by \((\exp(a), 0) \in G_0 \ltimes_\rho \mathbb{R}^N\), where \(a\neq 0\in \mathfrak h\) is a nontrivial element of a split Cartan subalgebra. With some abuse of notation, we denote by $\nu(f_0)=\nu(a)$ for a root or weight $\nu:\mathfrak h\to \R$. 

    The representation $\rho$ and cocompact lattice $\Gamma_0$ satisfying the condition \(\Gamma_0=\rho^{-1}(\mathrm{SL}_N(\mathbb{Z})\cap \rho(G_0))\) exist in abundance via standard arithmetic constructions. An explicit example will be provided in Section \ref{extwcf}.

    We say that a twisted semisimple homogeneous space \(X\) \emph{has a symplectic component} if there exists a non-zero skew-symmetric bilinear form \(\omega: \mathbb{R}^N \times \mathbb{R}^N \to \mathbb{R}\) that is invariant under the action, i.e., \(\omega(v_1, v_2) = \omega(\rho(g)v_1, \rho(g)v_2)\) for all \(g \in G_0\) and \(v_1, v_2 \in \mathbb{R}^N\).
\end{enumerate}

To unify notation, we denote \(G = G_0\), \(\Gamma = \Gamma_0\) in the semisimple homogeneous space case, and \(G = G_0 \ltimes_\rho \mathbb{R}^N\), \(\Gamma = \Gamma_0 \ltimes_\rho \mathbb{Z}^N\) in the twisted semisimple homogeneous space case.

By Theorem A of \cite{DWWX}, in both cases, since $f_0$ is a $K$-system, the smooth centralizer of \(f_0\) coincides with its algebraic centralizer consisting of products of left translations and automorphisms of $G$. 

Let $f_0$ be an element of the (twisted) Weyl chamber flow, and let $f$ be a $C^1$-small perturbation of $f_0$, then Theorem 3.1 of \cite{Wang} shows that the centralizer of $f$ is a Lie group, and the action $\mathcal{Z}(f)\times X\to X$ is jointly smooth. (See Theorem \ref{czlieg}.)

We define \(f_0\) to be a \emph{generic} element of the (twisted) Weyl chamber flow if the corresponding element in the Lie algebra avoids the Weyl chamber walls described in Section \ref{semisimple}. In other words, an element of the Weyl chamber flow is generic if all its roots are non-zero, and an element of the twisted Weyl chamber flow is generic if all its roots and weights are non-zero.

For \emph{generic} elements of the Weyl chamber flow on simple Lie groups, Wang \cite{Wang} showed that the centralizer of a perturbation \(f\) is a Lie group, and \(f\) is smoothly conjugate to an element of the Weyl chamber flow when the centralizer is of dimension at least two. Here, we establish an analogous rigidity result for generic elements in the twisted setting.

\begin{thm}\label{cenrigtwgen}
    Let \(X = (G_0 \ltimes_\rho \mathbb{R}^N) / (\Gamma_0 \ltimes_\rho \mathbb{Z}^N)\) be a twisted semisimple homogeneous space, where \(G_0\) is a real-split semisimple Lie group of rank $k$ with simple components of ranks \(k_i\). Suppose \(f_0\) is a generic element of the twisted Weyl chamber flow on \(X\). Let \(f\in \mathrm{Diff}^\infty(X)\) be a sufficiently small \(C^1\)-perturbation of \(f_0\). If \(\dim \mathcal Z(f)>k-\min_i k_i+1\), then \(\mathcal Z(f)\simeq \mathcal Z(f_0)\) and \(f\) is smoothly conjugate to an element of the twisted Weyl chamber flow.
\end{thm}

In particular, if $G_0$ is a real-split \emph{simple} Lie group, $f_0$ is a generic element of the twisted Weyl chamber flow, then for any $C^1$-small perturbation $f$ of $f_0$, either the centralizer has dimension $0$ or $1$, or $f$ is smoothly conjugate to an element of the twisted Weyl chamber flow.

It was shown in \cite{DWX23} that (under some additional conditions) if the dimension of the smooth centralizer of a diffeomorphism $f$ is the same as the manifold $X$, then up to a finite cover, $f$ is smoothly conjugate to a minimal translation on a torus. Theorem \ref{cenrigtwgen} proves local rigidity under a much weaker assumption on the dimension of the centralizer: the critical dimension that forces $f$ to be affine is smaller than the dimension of the center foliation of $f$, which is much smaller than the dimension of the ambient manifold.

We also obtain rigidity results for many \emph{non-generic} elements of the Weyl chamber flow. 

For semisimple homogeneous spaces or twisted semisimple homogeneous spaces, we have the following result. Here, for two Lie groups \(H_1,H_2\), we write \(H_1\doteq H_2\) if finite-index subgroups of \(H_1\) and \(H_2\) are isomorphic as Lie groups.

\begin{thm}\label{cenrignongen:twisted}
    Let \(X = G/\Gamma\) be a (twisted) semisimple homogeneous space.
    Let $f_0$ be a non-trivial element of the (twisted) Weyl chamber flow, such that the center of $\mathcal Z(f_0)$ has dimension at least two in each simple component of $G_0$.
    
    In the case that $G$ is a twisted product, we further assume either $f_0$ has no zero weights, or $\rho$ has no symplectic component.
    
    Let \(f\in \mathrm{Diff}^\infty(X)\) be a sufficiently small \(C^1\)-perturbation of \(f_0\). If \(\mathcal Z(f) \doteq \mathcal Z(f_0)\), then \(f\) is smoothly conjugate to an element of the (twisted) Weyl chamber flow.
\end{thm}

Since $\mathcal Z(f_0)$ agrees with the algebraic centralizer of $f_0$, the condition that the center of $\mathcal Z(f_0)$ has dimension at least two in each simple component of $G_0$ is a purely algebraic condition, which can also be restated as, in each simple factor $\mathfrak g_i$ of rank $k_i$, no collection of $k_i-1$ linearly independent roots all vanish on $f_0$.

\subsection{Some history and further questions}
The results above sit at the intersection of two rigidity programs.

The first is the centralizer rigidity program.  Smale's conjecture and the
theorem of Bonatti--Crovisier--Wilkinson \cite{BCW} show that trivial
centralizer is the generic behavior in the \(C^1\) topology.  Centralizer
rigidity studies the opposite situation: near algebraic models, a large
centralizer should force algebraic structure.  Damjanovi\'c--Wilkinson--Xu
\cite{DWX} initiated this program for partially hyperbolic systems such as
discretized geodesic flows and certain partially hyperbolic affine toral automorphisms.  Further
results were obtained for fibered partially hyperbolic systems and nilmanifold examples by the same authors and Sandfeldt \cite{DWX23, Sandfeldt}.  In the semisimple direction,
the author \cite{Wang} proved centralizer rigidity near generic elements of Weyl
chamber flows.

The present work extends this semisimple theory to twisted
Weyl chamber flows and to many non-generic elements.

The second is the local rigidity theory of higher-rank actions.  For
higher-rank abelian algebraic actions, Katok--Spatzier proved local
rigidity for many homogeneous Anosov actions \cite{KS}.  Damjanovi\'c and
Katok developed the geometric method for restrictions of diagonal actions of $\mathrm{SL}_n\R/\Gamma$
\cite{DK}.  Vinhage extended the theory to general semisimple homogeneous
spaces \cite{Vinhage}, and Vinhage--Wang treated a larger class of partially
hyperbolic homogeneous actions, including non-generic and twisted examples \cite{VinWang}.
Those results assume that a whole higher-rank action is perturbed.  The
present paper proves rigidity under a weaker, semi-local hypothesis: only
one element is perturbed, and the missing higher-rank action is recovered
from the centralizer.

From this viewpoint, the paper also fits naturally into the philosophy of
Zimmer-type rigidity.  Zimmer's program predicts that actions of
higher-rank groups and lattices are strongly constrained and often
algebraic.  Local rigidity results for standard higher-rank actions, such as
the work of Fisher--Margulis \cite{FM}, are part of this same philosophy.
Here the acting higher-rank group is not given at the outset after
perturbation.  Instead, the algebraic structure of the centralizer
plays the role of the higher-rank symmetry, and the conclusion is that the
single perturbed element is forced back into the algebraic family.

The main difficulty in the approach here is to find an appropriate element of the centralizer and
to prove that it has enough dynamical structure for higher-rank rigidity
methods to apply.  This is especially delicate in the non-generic case.
When the chosen Cartan element lies on a Weyl chamber wall, some roots or
weights vanish; the center direction becomes larger, and the usual
Lyapunov foliations may merge or degenerate.

The strategy is to use the centralizer itself to recover higher-rank
hyperbolic structure.  First, one identifies the affine model suggested by
the action of the centralizer to get dynamical description along the center leaves.  Then one uses the
interaction between the center foliation and the stable and unstable
foliations of \(f\) to prove that suitable centralizer elements have the
expected hyperbolicity along the coarse Lyapunov directions.  A technical
point is that these two pieces of information a priori produce two possibly
different affine models; part of the proof is to show that they coincide.

On the other hand, several questions remain open:

\begin{enumerate}
    \item Does Conjecture 2 of \cite{DWWX} hold for elements of the twisted Weyl chamber flow?

    \item In twisted semisimple homogeneous spaces, can we treat the cases where \(f_0\) has only one nonzero root up to sign, or where \(f_0\) possesses zero weights for general representations?

    \item  Can one extend the results to non-diagonalizable left translations in semisimple homogeneous spaces and twisted semisimple homogeneous spaces? 
    
    The main difference for these maps is that the dimension of the center foliation is usually greater than the dimension of the affine centralizer, which requires a different proof of partial hyperbolicity of elements in the centralizer.

    \item Finally, is it possible to obtain rigidity by only an assumption of the dimension of the centralizer, as seen for generic elements in semisimple homogeneous spaces of simple Lie groups? Is there a minimal critical dimension $r$, such that if $\mathrm{dim} \;\mathcal Z(f)> r$, then $f$ is smoothly conjugate to an element of the (twisted) Weyl chamber flow?
    
    We cannot expect the critical dimension to be $1$, as in the case of generic elements. For instance, consider \(f_0=\mathrm{diag}(e,e,\ldots,e,e^{-(n-1)})\) acting on \(\mathrm{SL}_n(\mathbb{R})/\Gamma\). A non-homogeneous perturbation of the form
    \[
    f(x)=\mathrm{diag}(e^{t(x)},e^{t(x)},\ldots,e^{t(x)},e^{-(n-1) t(x)})
    \]
    where \(t(x)\) is invariant under the left action of \(\mathrm{SO}(n-1)\) (viewed as a subgroup stabilizing the first \(n-1\) entries up to scalar), can lead to \(\dim \mathcal Z(f)\ge \dim \mathrm{SO}(n-1)+1\). This increase in centralizer dimension does not contradict Conjecture 2 of \cite{DWWX}, as the added dimension arises from a compact factor \(\mathrm{SO}(n-1)\) and a one-dimensional flow. It would be an interesting question to give an exact bound on the largest possible dimension of the centralizer of a non-affine perturbation.
\end{enumerate}

In a forthcoming paper with Wilkinson \cite{WangWil}, we also show that under some regularity condition, Conjecture 3 of \cite{DWWX} is true for non-generic diagonal elements with only one non-zero root in $\mathrm{SL}_n\R/\Gamma$, which treats the cases left out by Theorem \ref{cenrignongen:twisted}.

\subsection{Organization and outline of the proof}\label{outl}

We first recall some preliminaries about twisted Weyl chamber flow and partially hyperbolic dynamics in Section \ref{prelim}.

The remainder of the paper carries out the proof in the following steps.

First, we separate the centralizer $\mathcal Z(f)$ into the center-fixing part and a discrete part in $\mathrm{Out}(\Gamma)$. Since \(f\) is a partially hyperbolic diffeomorphism, elements of its centralizer \(\mathcal{Z}(f)\) preserve the stable, unstable, and center foliations. We define the \emph{center-fixing centralizer}
\[\mathcal{CZ}(f) := \{ g \in \mathcal{Z}(f) : g(\mathcal{W}^c_f(x)) = \mathcal{W}^c_f(x),\ \forall x \in X \}\] where $\W^c_f$ is the center foliation of $f$. In Section \ref{sec:czf=zf}, we establish that the quotient group \(\mathcal{Z}(f)/\CZ(f)\) is a subgroup of \(\mathrm{Out}(\Gamma)\) which is discrete. 

Next, we focus on the action of \(\mathcal{CZ}(f)\) on the universal cover of a center leaf \(\tilde{\mathcal{W}}^c_f(x)\). In Section \ref{sec:cla}, we show that this action is free and proper and establish many dynamical properties of the elements of the centralizer inside the center foliation of $f$. In particular, we show that there exists an appropriate model element $g_0\in \CZ_0(f_0)$, such that $g$ is $C^0$ close to $g_0$, and has the same type of dynamical behavior as $g_0$ when restricted to the center bundle.

 In Section \ref{topleaf}, we then show that \(g\) also exhibits the same essential dynamical properties as $g_0$ along the stable and unstable foliations of $f$—specifically, the expansion and contraction rates of certain topological foliations of $g$ are the same as $g_0$. 

Using these established dynamical properties, Section \ref{sec:cocrig} follows the established scheme for proving cocycle rigidity and local rigidity of higher-rank restrictions of the (twisted) Weyl chamber flow. This culminates in showing that \(f\) is smoothly conjugate to an element of the (twisted) Weyl chamber flow, thereby proving our main rigidity theorems.

\subsection{Acknowledgments}
The author thanks Amie Wilkinson for many helpful discussions and suggestions on an early draft. The author thanks Kurt Vinhage for explaining his work and for many useful discussions. The author also thanks Aaron Brown, David Fisher, Homin Lee, Sven Sandfeldt and Disheng Xu for helpful discussions and remarks.

\section{Setting and preliminaries}\label{prelim}
In this section, we provide a brief introduction to the twisted Weyl chamber flow and its dynamics, the theory of partially hyperbolic dynamical systems, Pesin theory, and some results on holonomies of cocycles.

\subsection{Twisted Weyl chamber flow}\label{semisimple}

Let $G_0$ be an $\mathbb{R}$-split semisimple Lie group with Lie algebra $\mathfrak{g}_0$, and let $\rho:G_0\to \mathrm{SL}_N(\mathbb{R})$ be a representation whose irreducible components are non-trivial. Let $\Gamma_0$ be a cocompact lattice in $G_0$ such that $\Gamma_0 = \rho^{-1}(\mathrm{SL}_N(\mathbb{Z}) \cap \rho(G_0))$. Since $G_0$ is $\mathbb{R}$-split, the adjoint action of elements in a Cartan subalgebra $\mathfrak{h} \subset \mathfrak{g}_0$ is diagonalizable over $\mathbb{R}$. We define the \emph{rank} of $G_0$ to be
\[
\operatorname{rank}(G_0) = \dim (\mathfrak{h}).
\]

We consider the semidirect product $G = G_0 \ltimes_\rho \mathbb{R}^N$ with multiplication
\[
(a,v) \cdot (b,w) = (ab, \; v + \rho(a)w).
\]
We take the lattice $\Gamma = \Gamma_0 \ltimes_\rho \mathbb{Z}^N$, which is cocompact in $G$, and consider the quotient space $X = G/\Gamma$.

Let $\mathfrak{g}$ be the Lie algebra of $G$ and $\mathfrak{h}$ be a Cartan subalgebra of $\mathfrak{g}_0$ (viewed as a subalgebra of $\mathfrak{g}$). We denote by $D = \exp(\mathfrak{h})$ the corresponding maximal split abelian subgroup of $G_0$.

The differential of the action of $D$ on $G$ (through the adjoint action on $\mathfrak{g}_0$ and the derived representation on $\mathbb{R}^N$) induces the following decomposition of the Lie algebra $\mathfrak{g}$ of $G$:
\[
\mathfrak{g} = \mathfrak{g}_0 \oplus \mathbb{R}^N = \mathfrak{h} \oplus \sum_{r \in \Delta_0} \mathfrak{g}_r \oplus \sum_{\nu \in \Phi_0} \mathfrak{e}_\nu.
\]
Here, 
\[
\mathfrak{g}_r = \{ u \in \mathfrak{g}_0 : \operatorname{ad}(h)u = r(h)u, \; \forall h \in \mathfrak{h} \}
\]
are the eigenspaces of the adjoint action of $\mathfrak{h}$ on $\mathfrak{g}_0$, and
\[
\Delta_0 = \{ r \in \mathfrak{h}^* : \mathfrak{g}_r \neq \varnothing, \; r \neq 0 \}
\]
is the set of roots of $G_0$. 

The weight spaces of $\rho$ are given by
\[
\mathfrak{e}_\nu = \{ v \in \mathbb{R}^N : \rho(\exp(h))v = e^{\nu(h)}v, \; \forall h \in \mathfrak{h} \},
\]
and the set of weights of $\rho $ is
\[
\Phi_0 = \{ \nu \in \mathfrak{h}^* : \mathfrak{e}_\nu \neq 0 \}.
\]

For $r \in \Delta_0$, we call $\ker(r)$ a \emph{Weyl chamber wall} of $G_0$; for $\mu \in \Delta_0 \cup \Phi_0$, we call $\ker(\mu)$ a \emph{Weyl chamber wall} of $G$. The connected components of
\[
\mathfrak{h} \setminus \bigcup_{\mu \in \Delta_0 \cup \Phi_0} \ker(\mu)
\]
are called the \emph{Weyl chambers} of $G$.

The \emph{twisted Weyl chamber flow} in $X=(G_0\ltimes \R^N)/(\Gamma_0\ltimes \Z^N)$ is the action of $D\ltimes \{0\}$ acting on $X$ via left multiplication. An element $f_0$ of the twisted Weyl chamber flow is an element of the form $f_0=(\exp(a), 0)$, for $a \in \mathfrak{h}$. We say that such an element $f_0$ is \emph{generic} if $a \notin \ker(\mu)$ for any $\mu \in \Delta_0 \cup \Phi_0$.

The differential of an element $f_0 = (\exp(a), 0)$ of the twisted Weyl chamber flow (where $a \in \mathfrak{h}$) preserves a splitting of the tangent space $T_x X$ at each point $x \in X$, whose components correspond to the Lie algebra decomposition:
\[
\mathfrak{g} = \mathfrak{g}_0 \oplus \mathbb{R}^N = \mathfrak{h} \oplus \sum_{r \in \Delta_0} \mathfrak{g}_r \oplus \sum_{\lambda \in \mathbb{R}} \mathfrak{e}_{\lambda},
\]
where $\mathfrak{e}_{\lambda} = \bigoplus_{\nu \in \Phi_0 : \nu(a) = \lambda} \mathfrak{e}_\nu$.
The contraction and expansion rates (Lyapunov exponents) of $Df_0$ on these subbundles are given by $r(a)$ for $\mathfrak{g}_r$ and $\lambda$ for $\mathfrak{e}_\lambda$. Specifically, for $k \in \mathbb{Z}$, we have $|Df_0^k|_{\mathfrak{g}_r}| = e^{k r(a)}$ and $|Df_0^k|_{\mathfrak{e}_\lambda}| = e^{k \lambda}$. This shows that $f_0$ is a partially hyperbolic (see Section \ref{prel:ph}) diffeomorphism  in $X$. The center bundle corresponds to the sum of eigenspaces where the exponent is zero, specifically $E^c_{f_0}\cong \mathfrak{h} \oplus_{r(a)=0} \mathfrak g_r\oplus \mathfrak{e}_0$.

The integral manifold through $x$ for the center bundle forms the center leaf $\mathcal{W}^c_{f_0}(x)$. This forms a center foliation $\W^c_{f_0}$, whose leaves are cosets of
the  subgroup \[
G^c = \exp(\mathfrak{h} \oplus_{r(a)=0} \mathfrak g_r\oplus \mathfrak{e}_0).\] The left cosets of $G^c$ in $G$ project to form the center leaves of the partially hyperbolic diffeomorphism $f_0$.

\subsubsection{A concrete example of a twisted semisimple homogeneous space}\label{extwcf}
We now give a concrete example of a twisted semisimple homogeneous space with a cocompact lattice satisfying our assumptions for $G = \mathrm{SL}_3(\mathbb{R}) \ltimes_\rho \mathbb{R}^9$, as described in \cite{Witte}.

Let $G_0 = \mathrm{SL}_3(\mathbb{R})$, with its Lie algebra
\[
\mathfrak{sl}_3(\mathbb{R}) = \{ A \in M_{3 \times 3}(\mathbb{R}) : \operatorname{tr} A = 0 \}.
\]
We identify $\mathbb{R}^9$ with the vector space of $3 \times 3$ real matrices, denoted $M_{3 \times 3}(\mathbb{R})$.

Let $\rho_0: \mathrm{SL}_3(\mathbb{R}) \to \mathrm{GL}_9(\mathbb{R})$ be the representation defined by left matrix multiplication: $\rho_0(A)(B) = A \cdot B$ for $A \in \mathrm{SL}_3(\mathbb{R})$ and $B \in M_{3 \times 3}(\mathbb{R})$. Since $\det(A)=1$, the determinant of the linear transformation $\rho_0(A)$ (viewed as an element of $\mathrm{GL}_9(\mathbb{R})$) is also 1, so $\rho_0(A) \in \mathrm{SL}_9(\mathbb{R})$.

The construction of the cocompact lattice $\Gamma_0 \subset \mathrm{SL}_3(\mathbb{R})$ follows Proposition 6.7.4 in \cite{Wittebook}. Let $L$ be a cubic Galois extension of $\mathbb{Q}$ with ring of integers $\mathcal{O}$. Let $\sigma$ be the generator of $\operatorname{Gal}(L/\mathbb{Q})$, and let $p \in \mathbb{N}_+$ be an integer such that $p \neq t \sigma(t) \sigma^2(t)$ for any $t \in L$.

Define a map $\psi: L^3 \to M_{3 \times 3}(L)$ by
\[
\psi(x,y,z) = \begin{pmatrix}
     x & y & z \\
     p\sigma(z) & \sigma(x) & \sigma(y) \\
     p\sigma^2(y) & p\sigma^2(z) & \sigma^2(x)
 \end{pmatrix}.
\]
The lattice $\Gamma_0$ is then defined as $\Gamma_0=\mathrm{SL}_3(\mathbb{R}) \cap \psi(\mathcal{O}^3)$. It can be shown that $\Gamma_0$ is a cocompact lattice in $\mathrm{SL}_3(\mathbb{R})$ and can equivalently be characterized as the set of $g \in \mathrm{SL}_3(\mathbb{R})$ such that $g \psi(\mathcal{O}^3) = \psi(\mathcal{O}^3)$.

Now, we define the representation $\rho$ that ensures the integrality condition for the twisted lattice. Choose $T \in \mathrm{GL}_9(\mathbb{R})$ such that $T(\mathbb{Z}^9) = \psi(\mathcal{O}^3)$ (such a $T$ exists because both are lattices in $\R^9$). Define $\rho: \mathrm{SL}_3(\mathbb{R}) \to \mathrm{SL}_9(\mathbb{R})$ by $\rho(g)B = T^{-1} \rho_0(g) T B$ for any $B \in \mathbb{R}^9$. With this choice, we have
\[
\Gamma_0 = \{ g \in \mathrm{SL}_3(\mathbb{R}) : \rho_0(g) T(\mathbb{Z}^9) = T(\mathbb{Z}^9) \} \]\[= \{ g \in \mathrm{SL}_3(\mathbb{R}) : T^{-1} \rho_0(g) T (\mathbb{Z}^9) = \mathbb{Z}^9 \} = \rho^{-1}(\mathrm{SL}_9(\mathbb{Z})).
\]
This construction yields a non-trivial twisted semisimple homogeneous space $(\mathrm{SL}_3(\mathbb{R}) \ltimes_\rho \mathbb{R}^9) / (\Gamma_0 \ltimes_\rho \mathbb{Z}^9)$ with a cocompact lattice.

A standard choice of a Cartan subalgebra in $\mathfrak{sl}_3(\mathbb{R})$ is the set of diagonal matrices:
\[
\mathfrak{h} = \left\{ \operatorname{diag}(t_1, t_2, t_3) : t_1 + t_2 + t_3 = 0, \; t_1, t_2, t_3 \in \mathbb{R} \right\}.
\]
The roots for the adjoint action on $\mathfrak{sl}_3(\mathbb{R})$ are $r_{ij}(\operatorname{diag}(t_1,t_2,t_3)) = t_i - t_j$, each corresponding to a 1-dimensional root space.
For the representation $\rho_0(A)(B) = A \cdot B$, the weights are $\mu_i(\operatorname{diag}(t_1,t_2,t_3)) = t_i$, each associated with a 3-dimensional weight space (corresponding to the $i$-th row of $B$).

The Weyl chamber walls in $\mathfrak{h}$ are given by the equations $t_i = t_j$ (from the roots) and $t_i = 0$ (from the weights). Since $\mathfrak{h}$ is a 2-dimensional space (isomorphic to $\mathbb{R}^2$ embedded in $\mathbb{R}^3$), these six lines partition $\mathfrak{h}$ into $12$ distinct Weyl chambers.

\subsubsection{Universal central extension of a twisted semisimple homogeneous space}\label{cenext}
In this section we define the universal central extension, which will be useful in Section \ref{sec:cocrig}, where we apply some algebraic properties of the Weyl chamber flow lifted to the universal central extension. 

Let $G = G_0 \ltimes_\rho \mathbb{R}^N$, where $G_0$ is a semisimple Lie group and $\rho$ is a representation with no trivial component. Since $G_0$ is semisimple, its Lie algebra $\mathfrak{g}_0$ is perfect (i.e., $[\mathfrak{g}_0, \mathfrak{g}_0] = \mathfrak{g}_0$). The Lie algebra of $G$ is $\mathfrak{g} = \mathfrak{g}_0 \oplus \mathbb{R}^N$, with the Lie bracket defined as $[(X,v), (Y,w)] = ([X,Y], d\rho(X)w - d\rho(Y)v)$. Because $G_0$ is semisimple and $\rho$ has no trivial component, the derived representation satisfies $d\rho(\mathfrak{g}_0)\mathbb{R}^N=\mathbb{R}^N$. Consequently, the commutator subalgebra $[\mathfrak{g},\mathfrak{g}]$ equals $\mathfrak{g}$, which implies that $\mathfrak{g}$ is perfect.

\begin{defi}
    Let $S$ be a group. A \emph{central extension} of $S$ is a group homomorphism $\theta: S' \to S$ such that its kernel $\ker(\theta)$ is contained in the center of $S'$. A central extension is \emph{perfect} if the covering group $S'$ is perfect.
\end{defi}

For a semisimple Lie group $S$, all perfect central extensions are factors of its universal covering group. However, for the semidirect product $G = G_0 \ltimes_\rho \mathbb{R}^N$, non-trivial perfect central extensions exist if and only if the representation $\rho$ possesses a symplectic component (i.e., admits a non-trivial $\rho$-invariant anti-symmetric bilinear form on $\mathbb{R}^N$); see Proposition 5.6 of \cite{VinWang}.

\begin{prop}[Definition 5.2, Proposition 5.7, \cite{VinWang}]
    There exists a perfect central extension $\Theta: \widetilde{G} \to G$, such that
    \[
    \widetilde{G} = \widetilde{G}_0 \ltimes_{\widetilde{\rho}} N,
    \]
    where $\widetilde{G}_0$ is the universal cover of $G_0$ and $N$ is a nilpotent Lie group whose Lie algebra is $\mathbb{R}^N \oplus \Omega_\rho$. Here, $\Omega_\rho$ denotes the vector space of $\rho$-invariant anti-symmetric bilinear forms in $\mathbb{R}^N$.
    
    Any perfect central extension $H \to G$ is a factor of this extension $\Theta: \widetilde{G} \to G$.
\end{prop}

We refer to $\widetilde{G}$ as the \emph{universal central extension} of $G$.

The group $G$ admits non-trivial perfect central extensions if and only if $\rho$ has a symplectic component (i.e., $\Omega_\rho \neq \{0\}$). In this case, starting from the lattice $\Gamma = \Gamma_0 \ltimes_\rho \mathbb{Z}^N$ in $G$, we can construct a corresponding lattice $\widetilde{\Gamma}$ in $\widetilde{G}$. Specifically, $\widetilde{\Gamma}$ is a central extension of $\Gamma$ by a lattice in $\Omega_\rho$, such that $\Theta(\widetilde{\Gamma}) = \Gamma$.

\subsubsection{Some unified notation}

In this section, we list some notations that will be used consistently throughout this paper.

\begin{enumerate}
    \item We denote by $G_0$ a semisimple, $\mathbb{R}$-split Lie group with no simple factors of rank 1. $\Gamma_0$ denotes a cocompact lattice in $G_0$.

    \item The ambient group $G$ is taken to be either $G_0$ (for the semisimple homogeneous space case) or the semidirect product $G_0\ltimes_\rho \mathbb{R}^N$ (for the twisted semisimple homogeneous space case). Here, $\rho: G_0 \to \mathrm{SL}_N(\mathbb{R})$ is a representation with no trivial component, such that $\Gamma_0=\rho^{-1}( \mathrm{SL}_N(\mathbb{Z})\cap \rho(G_0))$. Correspondingly, $\Gamma$ denotes either $\Gamma_0$ or $\Gamma_0\ltimes_\rho \mathbb{Z}^N$. The manifold under consideration is $X=G/\Gamma$.

    \item We use $\Delta_0$ to denote the set of roots of $G_0$ (associated with the adjoint action of $\mathfrak{h}$ on $\mathfrak{g}_0$). $\Phi_0$ denotes the set of weights of the representation $\rho$ (associated with the action of $\mathfrak{h}$ on $\mathbb{R}^N$).
    The unified set of roots and weights, $\Lambda_0$, is defined as $\Delta_0$ in the semisimple homogeneous space case (i.e., when $\mathbb{R}^N=\{0\}$), and as $\Delta_0\cup \Phi_0$ in the twisted semisimple homogeneous space case.

    \item For $\mu\in \Lambda_0$, we shall  denote by $V_\mu=\mathfrak g_\mu\oplus \mathfrak e_\mu$ the eigenspace of $\mu\in \mathfrak h^*$, with the convention that \(\mathfrak g_\mu=0\) if \(\mu\notin\Delta_0\),
and \(\mathfrak e_\mu=0\) if \(\mu\notin\Phi_0\).

    \item Let $A \subset \mathfrak{h}$ be an additive subgroup of the Cartan subalgebra. We may consider the root and weight systems restricted to the action of $A$. We denote by $\Delta_A$ the roots of the action of $A$ on $\mathfrak{g}_0$, and by $\Phi_A$ the weights of the action of $A$ on $\mathbb{R}^N$.
    Similarly, $\Lambda_A$ is defined as $\Delta_A$ for semisimple homogeneous spaces and $\Delta_A \cup \Phi_A$ for twisted semisimple homogeneous spaces.

    \item $\widetilde{G}$ denotes the universal cover of $G$ if $G$ is a semisimple Lie group. In the twisted case, $\widetilde{G}$ denotes the universal central extension of $G=G_0\ltimes_\rho \mathbb{R}^N$.
\end{enumerate}
\subsection{Partial hyperbolicity and regularity}\label{prel:ph}
We recall the definitions of dominated splitting and partial hyperbolicity, and introduce the concept of leaf conjugacy, drawing on results from \cite{HPS}.

\begin{defi}[Dominated Splitting]
Let $M$ be a Riemannian manifold, and let $f:M\to M$ be a diffeomorphism. A \emph{dominated splitting} of $f$ is a $Df$-invariant decomposition $TM=E^1\oplus E^2\oplus \cdots \oplus E^k$ satisfying:
\begin{itemize}
\item For each $1\le i\le k$, $Df(E^i(x))=E^i(f(x))$ for any $x\in M$.
\item There exist constants $C>0$ and $0 < \lambda < 1$ such that for any integer $n \ge 1$, $x\in M$, and for any unit vectors $v\in E^i(x), u\in E^{i+1}(x)$,
$$
\|Df^n(x)v\|\le C\lambda^n \|Df^n(x)u\|,
$$
for $1\le i\le k-1$.
\end{itemize}
\end{defi}

\begin{defi}[Partially Hyperbolic Diffeomorphism]\label{defi:ph}
A diffeomorphism $f:M\to M$ of a Riemannian manifold $M$ is \emph{partially hyperbolic} if it admits a \emph{dominated splitting} $TM=E^s\oplus E^c\oplus E^u$, and there exist constants $C>0$ and $0<\mu_u,\mu_s<1$ such that for any integer $n \ge 1$, and any $x\in M$:
$$
\|Df^n(x)v\|\le C\mu_s^n\|v\|,\qquad \text{for any } v\in E^s(x),
$$
and
$$
\|Df^{-n}(x)u\|\le C\mu_u^n\|u\|,\qquad \text{for any } u\in E^u(x).
$$
In this paper, we assume that both the stable and unstable bundles $E^s$ and $E^u$ are non-trivial (i.e., non-zero dimensional).
\end{defi}

For a partially hyperbolic diffeomorphism, the stable and unstable bundles $E^s$ and $E^u$ are uniquely integrable to $f$-invariant stable and unstable foliations $\mathcal{W}^s$ and $\mathcal{W}^u$, respectively. The leaves of these foliations are globally defined by:
$$
\mathcal{W}^s(x)=\{y\in M:\; \exists\; C>0,\; d(f^n(x),f^n(y))\le C\mu_s^n \,\; \forall\; n>0 \},
$$
and
$$
\mathcal{W}^u(x)=\{y\in M: \;\exists\; C>0, \;d(f^n(x),f^n(y))\le C\mu_u^{-n} , \;\forall\; n<0 \}.
$$
Unlike $E^s$ and $E^u$, the center bundle $E^c$ is generally not integrable to a foliation. A partially hyperbolic diffeomorphism is said to be normally hyperbolic if there exists an $f$-invariant center foliation $\mathcal{W}^c$ such that $T\mathcal{W}^c = E^c$.

\begin{defi}[Normally Hyperbolic]
Let $f:M\to M$ be a diffeomorphism. For $r\ge 1$, we say $f$ is \emph{$r$-normally hyperbolic} with respect to a foliation $\mathcal{F}$ if $f$ preserves $\mathcal{F}$, $f$ is partially hyperbolic with respect to $E^c=T\mathcal{F}$, and there exists an integer $k \ge 1$ such that: 
$$
\sup_{x\in M}\|D_xf^k|_{E^s}\|\cdot \|(D_xf^k|_{E^c})^{-1}\|^r<1
$$
and
$$
\sup_{x\in M}\|(D_xf^k|_{E^u})^{-1}\|\cdot \|D_xf^{k}|_{E^c}\|^r<1.
$$

We say that $f$ is \emph{normally hyperbolic} if it is $1$-normally hyperbolic.
\end{defi}

Note that $r$-normal hyperbolicity is a $C^1$-open condition.

For a normally hyperbolic diffeomorphism $(f,\mathcal{F})$, heuristically, the induced action of $f$ on the leaf space $M/\mathcal{F}$ can be viewed as an Anosov system. Consequently, one might expect structural stability of $f$ up to the leaves of $\mathcal{F}$. Under suitable conditions, perturbations of $f$ are $C^0$-conjugate to $f$ modulo the leaves of $\mathcal{F}$; we call such a perturbation \emph{leaf conjugate} to $f$. Now we define leaf conjugacy precisely.

\begin{defi}[Leaf Conjugacy]
Suppose $(f,\mathcal{F}_f)$ and $(g,\mathcal{F}_g)$ are two diffeomorphisms of $M$ with invariant foliations $\mathcal{F}_f$ and $\mathcal{F}_g$, respectively. Then $(f,\mathcal{F}_f)$ and $(g,\mathcal{F}_g)$ are said to be \emph{leaf conjugate} via a leaf conjugacy $h\in \mathrm{Homeo}(M)$ if
$$
h(\mathcal{F}_f(x))=\mathcal{F}_g(h(x)) \quad \text{for all } x\in M.
$$
\end{defi}

A notion slightly stronger than normal hyperbolicity is dynamical coherence.

\begin{defi}[Dynamically Coherent]
A partially hyperbolic diffeomorphism $f$ is said to be \emph{dynamically coherent} if there exist two $f$-invariant foliations $\mathcal{W}^{cs}$ and $\mathcal{W}^{cu}$ that are tangent to the continuous subbundles $E^{cs}=E^c\oplus E^s$ and $E^{cu}=E^c\oplus E^u$, respectively.

If $f$ is dynamically coherent, the center foliation of $f$, denoted $\mathcal{W}^c$, is defined as the intersection $\mathcal{W}^{cs}\cap \mathcal{W}^{cu}$.
\end{defi}

By definition, $\mathcal{W}^c$ is an $f$-invariant foliation tangent to $E^c$ at every point, which implies that $(f,\mathcal{W}^c)$ is normally hyperbolic. It is important to note that dynamical coherence does not require $E^{cs}$, $E^{cu}$, or $E^c$ to be uniquely integrable to their respective foliations. For a more in-depth discussion, see \cite{BWcoh}.

Now we introduce a central result about leafwise structural stability for normally hyperbolic diffeomorphisms. A more general statement and proof can be found in \cite{BWcoh} and \cite{PSWholrev}.

\begin{thm}[Leafwise structural stability]\label{leafconj}
Let $f$ be a $C^r$ diffeomorphism of $M$ that admits an $f$-invariant center foliation $\mathcal{F}$ (i.e., $(f,\mathcal{F})$ is normally hyperbolic). If $(f,\mathcal{F})$ is $r$-normally hyperbolic, then the leaves of $\mathcal{F}$ are uniformly $C^r$. Furthermore, the bundles $E^u$ and $E^s$ are uniquely integrable with leaves as smooth as $f$, and $f$ is $C^1$-stably dynamically coherent.

Moreover, if $f$ is plaque expansive, e.g. if $\W^c_f$ is $C^1$ or if $f|_{\W^c_f}$ is isometric, then there exists a bi-H\"older leaf conjugacy $\phi$ from $f$ to any $C^1$-perturbation of $f$, such that $\phi$ is uniformly $C^r$ when restricted to the center leaves $\mathcal{W}^c_f$.
\end{thm}

Applying this to the diffeomorphisms of $X = G/\Gamma$ that we are considering, we see that any $C^1$-small perturbation $f$ of an element of the (twsited) Weyl chamber flow $f_0$ inherits desirable dynamical properties.

\begin{prop}\label{fparhyp}
Let $f_0$ be a non-trivial element of the (twisted) Weyl chamber flow. Then $f_0$ is $r$-normally hyperbolic with respect to its center foliation for any $r\in \N$. Consequently, for each fixed $r\in \N$ and any sufficiently small $C^1$-perturbation $f \in \mathrm{Diff}^r(X)$ of $f_0$:
\begin{enumerate}
    \item $f$ is a partially hyperbolic diffeomorphism with non-trivial stable and unstable bundles.
    \item The stable and unstable bundles $E^u_f,E^s_f$ are uniquely integrable.
    \item $f$ is dynamically coherent and $r$-normally hyperbolic.
    \item There exists a bi-H\"older leaf conjugacy between $f_0$ and $f$ that is uniformly $C^r$ in the center leaves.
\end{enumerate}
\end{prop}
\begin{proof}
As established in Section \ref{semisimple}, a non-trivial element $f_0$ of the (twisted) Weyl chamber flow induces both positive and negative roots, thus possessing non-trivial stable and unstable bundles. Therefore, by the splitting of the Lie algebra, $f_0$ is $r$-normally hyperbolic with respect to the foliation $\mathcal{W}^c_{f_0}$ given by the cosets of $G^c$. Applying Theorem \ref{leafconj} directly yields the desired properties for $f$.
\end{proof}
\subsubsection{Accessibility}
Now we recall some definitions and facts about $su$-holonomies of partially hyperbolic diffeomorphisms.

An \emph{$su$-path} of $f$ is a piecewise $C^1$ curve $\gamma:[0,1]\to M$ such that each segment of the curve is contained within a stable or unstable leaf. More precisely, there exists an integer $k \ge 1$ and points $0=t_0<t_1<t_2<\cdots<t_k=1$ such that for each $i \in \{0, \dots, k-1\}$, $\gamma([t_i,t_{i+1}])$ is contained in either $\mathcal{W}^s(\gamma(t_i))$ or $\mathcal{W}^u(\gamma(t_i))$. Such an $su$-path is said to be $k$-legged. We sometimes denote an $su$-path by its sequence of intermediate points, e.g., $[\gamma(0),\gamma(t_1),\dots,\gamma(t_k)]$.

\begin{defi}[Accessibility]
A partially hyperbolic diffeomorphism $f$ is said to be \emph{accessible} if for any two points $x,y\in M$, there exists an $su$-path from $x$ to $y$.
\end{defi}

For partially hyperbolic systems with center dimension $1$ or $2$, accessibility is known to be an open property (see \cite{Didier} and \cite{AV}). Stable accessibility (meaning accessibility persists under $C^1$-small perturbations) has also been shown to be $C^r$-dense for $r\ge 1$ for center dimension $1$ in \cite{RHU}, and $C^1$-dense for arbitrary center dimension in \cite{DW}. Furthermore, if the center foliation of a partially hyperbolic diffeomorphism $f$ is smooth, then accessibility is $C^1$-open around $f$ by Proposition 1.4 of \cite{GPS}.

For our specific case, it can be shown that $f_0$ is accessible due to its underlying Lie algebra structure; see Corollary 9.10 of \cite{VinWang}. Since the center foliation is (transversely) smooth for $f_0$, we can apply Proposition 1.4 in \cite{GPS} to conclude that accessibility is an open property around $f_0$.

\begin{prop}\label{faccess}
Let $f_0$ be a non-trivial element of the (twisted) Weyl chamber flow in $X=G/\Gamma$. Then any small enough $C^1$-perturbation $f\in \mathrm{Diff}^\infty(X)$ of $f_0$ is accessible. Furthermore, there exist universal constants $C,N$ that depend only on $X$, such that for any $x,y\in X$, there exists an $f$-$su$-path of at most $N$ legs and of length $\le C$ connecting $x$ and $y$.
\end{prop}

\subsubsection{Center-bunching and ergodicity}
We now recall the property of center-bunching and its relation with ergodicity of a partially hyperbolic system.

\begin{defi}[Center Bunching]
We say that a partially hyperbolic diffeomorphism $f:M\to M$ is \emph{center $r$-bunched} if $f$ is r-normally hyperbolic, and there exists an integer $k \ge 1$ such that the following conditions hold:
\begin{align*}
\sup_{x\in M} \|D_xf^k|_{E^s}\|\cdot \|(D_xf^k|_{E^c})^{-1}\|\cdot \|D_xf^k|_{E^c}\|^r &<1, \text{ and } \\
\sup_{x\in M} \|(D_xf^k|_{E^u})^{-1}\|\cdot \|(D_xf^k|_{E^c})^{-1}\|^r\cdot \|D_xf^k|_{E^c}\| &<1.
\end{align*}
We say that $f$ is \emph{center bunched} if it is center $1$-bunched.
\end{defi}

The following theorem of Burns and Wilkinson \cite{BWerg} shows that partial hyperbolicity, center bunching, and accessibility together imply ergodicity.

\begin{thm}[Theorem 0.1, \cite{BWerg}]\label{accerg}
Let $f$ be a $C^2$, partially hyperbolic,  diffeomorphism that is center bunched and accessible, and suppose $f$ preserves a measure $\mu$ that is equivalent to the Lebesgue measure. Then $f$ is ergodic with respect to $\mu$.
\end{thm}

 \subsection{Holonomy of a foliation}

Let $M$ be a smooth manifold. Let $\mathcal{F}$ be a foliation in $M$ which is a decomposition of $M$ into a collection of disjoint, connected, $C^1$-embedded submanifolds, called \emph{leaves}. For each $x \in M$, we denote the unique leaf containing $x$ by $\mathcal{F}(x)$. 

For a leaf $\mathcal{F}(x)$, we denote by $\mathcal{F}(x,R) = \{ y \in \mathcal{F}(x) : d_{\mathcal{F}(x)}(x,y) < R \}$ a bounded portion of the leaf. A \emph{flow box} of $\mathcal{F}$ is a foliated region $B \subset M$ that is locally homeomorphic to a product space, such that the leaves of $\mathcal{F}$ within $B$ correspond to one of the product factors.

A \emph{transversal} of $\mathcal{F}$ is an embedded disk $S \subset M$ such that $S$ intersects every leaf of $\mathcal{F}$ transversally. 

For two transversals $S_1, S_2$ and a flow box $B$ that intersects both, the \emph{local holonomy map} $h_{S_1, S_2}^{\mathcal{F}, B}: S_1 \cap B \to S_2$ is defined by following the plaques (leaf segments) within $B$. For general transversals $S_1, S_2$, the \emph{holonomy map} $h_{S_1, S_2}^{\mathcal{F}}: \mathrm{Dom}(h) \to S_2$ is defined on the subset $\mathrm{Dom}(h) \subset S_1$ consisting of points whose leaves also intersect $S_2$. This map is constructed by composing a finite chain of local holonomies across flow boxes covering a leafwise path between $S_1$ and $S_2$; the resulting map depends only on the homotopy class of the path within the leaf.

\subsubsection{Smoothness of $su$-holonomy}\label{sec:suhol}
In this paper, we mainly consider stable and unstable holonomies, with the center leaves serving as sections.

Suppose $f$ is a dynamically coherent, normally hyperbolic diffeomorphism with center foliation $\mathcal{W}^c$. The stable holonomy maps between center leaves are defined by following stable leaves. Specifically, for $p \in M$ and $q \in \mathcal{W}^s_f(p)$, consider a local product structure where stable leaves are transverse to center leaves. Then there exist sufficiently small neighborhoods $U(p) \subset \mathcal{W}^c_f(p)$ of $p$ and $U(q) \subset \mathcal{W}^c_f(q)$ of $q$, and a constant $R > 0$ (e.g., $R = 2 d_{\mathcal{W}^s_f}(p,q)$), such that for every $x \in U(p)$, the stable leaf segment $\mathcal{W}^s_f(x, R) = \{ y \in \mathcal{W}^s_f(x) : d_{\mathcal{W}^s_f}(x,y) < R \}$ intersects $U(q)$ in exactly one point. The \emph{local stable holonomy} is then defined by
$$
h^s_{f,p,q}: U(p) \to U(q), \qquad x \mapsto U(q) \cap \mathcal{W}^s_f(x,R).
$$
Local unstable holonomies are defined analogously. These stable and unstable holonomies along center leaves are known to possess significant smoothness properties.

\begin{thm}[Theorem B, \cite{PSW}; Theorem 1.1, Section 3.6 \cite{Saghin}]\label{c1hol}
Let $f \in \mathrm{Diff}^{r+1}(M)$, $r \ge 1$, be a normally hyperbolic, center $r$-bunched diffeomorphism of a Riemannian manifold $M$. Then the local stable and unstable holonomy maps between center leaves are uniformly $C^r$.  Moreover, the holonomy $h^s_{f,p,q}$ (where $q \in \mathcal{W}^s_f(p)$) depends $C^1$-continuously on $f$ in the $C^1$ topology on $\mathrm{Diff}^1(M)$, and depends $C^r$-continuously on the base points $p,q$ in the $C^0$ topology on $M$, and there is a uniform bound $\|Dh^s_{f,p,q}-\mathrm{Id}\|\le C_f\cdot d(p,q)$ where the constant $C_f$ only depends on $f$. Similar results hold for the unstable holonomy.
\end{thm}

We may also define the \emph{holonomy of an $su$-path}. Given an $su$-path $\gamma = [x_0, x_1, \dots, x_k]$ of $f$ (where $x_{i+1} \in \mathcal{W}^{t(i)}_f(x_i)$ for $t(i) \in \{s, u\}$), we define the local \emph{$su$-holonomy} on a neighborhood of $x_0$ in $\mathcal{W}^c_f(x_0)$ by concatenating the sequence of stable and unstable holonomies between center leaves along $\gamma$. More concretely,
$$
h_\gamma^f = h^{t(k-1)}_{f,x_{k-1},x_k} \circ h^{t(k-2)}_{f,x_{k-2},x_{k-1}} \circ \cdots \circ h^{t(0)}_{f,x_0,x_1}: U(x_0) \to U(x_k),
$$
where $U(x_i) \subset \mathcal{W}^c_f(x_i)$ are sufficiently small neighborhoods of $x_i$ (and $x_0$ is the starting point of the path).

Let $f'$ be a sufficiently small $C^1$-perturbation of $f$, and let $\phi$ be a bi-H\"older leaf conjugacy from $f$ to $f'$. For any $su$-path $\gamma = [x_0, x_1, \dots, x_k]$ of $f$ with bounded length and number of legs, the conjugacy $\phi$ naturally induces an $su$-path $\gamma' = [\phi(x_0), \phi(x_1), \dots, \phi(x_k)]$ for $f'$, where $\phi(x_{i+1}) \in \mathcal{W}^{t(i)}_{f'}(\phi(x_i))$. As a direct application of Theorem \ref{c1hol} and the smoothness of $\phi$ along center leaves, the induced $su$-holonomies $h^{f'}_{\gamma'}$ are also smooth and depend smoothly on $f$ and $\gamma$.

\subsubsection{Existence of global holonomy}
We say that $f$ has \emph{global holonomy} in a center leaf $\mathcal{W}^c_f(x_0)$ if, for every $su$-path $\gamma$ of $f$ starting with $x_0 \in \mathcal{W}^c_f(x_0)$ and ending at $x_k \in \mathcal{W}^c_f(x_k)$, the local holonomy map $h_\gamma^f$ can be uniquely extended to a map defined on the entire leaf $\mathcal{W}^c_f(x_0)$. More precisely, this means that for any such $\gamma$, there exists a family of local $su$-holonomies $h_{\gamma_j}^f: U_j \to \mathcal{W}^c_f(x_k)$ where the small neighborhoods $\{U_j\}$ form an open cover of $\mathcal{W}^c_f(x_0)$, and these local holonomies agree on overlaps, thereby defining a global smooth map $h_\gamma^f: \mathcal{W}^c_f(x_0) \to \mathcal{W}^c_f(x_k)$.

In the case where $f$ is a $C^1$-perturbation of a partially hyperbolic affine model, the following result shows that $f$ always has global holonomy in the universal cover.

\begin{prop}[Proposition 2.8 \cite{Wang}]
    Let $f_0$ be a partially hyperbolic affine diffeomorphism on a compact homogeneous space $G/\Gamma$, then any sufficiently small $C^1$-perturbation $f$ of $f_0$ has global holonomy in the universal cover. In particular, $f$ has global holonomy in a center leaf $\W^c_f(x_0)$ if $\W^c_f(x_0)$ is simply connected.
\end{prop}

\subsection{Some Pesin theory}
This section reviews the Oseledets splitting theorem and fundamental concepts from Pesin theory that are utilized in this paper.

\begin{thm}[Oseledets Splitting Theorem]\label{Oseled}
Let $f\in \mathrm{Diff}^2(M)$ be a diffeomorphism on a compact Riemannian manifold $M$ that preserves an invariant Borel measure $\mu$. Then for $\mu$-almost every $x\in M$, there exists a unique measurable and $Df$-invariant decomposition $T_xM=\bigoplus_{i=1}^{m(x)} H^i_f(x)$ such that for every non-zero vector $v\in H^i_f(x)$, the \emph{Lyapunov exponent} $\chi_i(x)$ given by
\[
\chi_i(x)=\lim_{|k|\to \infty} \frac{1}{k} \log \frac{\|Df^k v\|}{\|v\|}
\]
exists and is independent of $v$.
\end{thm}

\begin{rmk}
If $f$ is further assumed to be ergodic, then the Lyapunov exponents $\chi_i$, the dimensions $\dim H^i_f(x)$, and the number of distinct exponents $m(x)=m$ are almost everywhere constant in $M$. In this case, we can order the exponents as $\chi_1 < \chi_2 < \cdots < \chi_m$.

The Oseledets splitting theorem generalizes to arbitrary measurable cocycles (not just the derivative cocycle) and to abelian group actions; for a detailed exposition, see \cite{introdyn}, Supplement 2.
\end{rmk}

Based on the Oseledets splitting of a diffeomorphism $f$, we define the subbundles corresponding to negative and positive exponents:
\[
E^-_f(x)=\bigoplus_{\chi_i(x)<0} H^i_f(x), \qquad
E^+_f(x)=\bigoplus_{\chi_i(x)>0} H^i_f(x).
\]
For $\mu$-almost every regular point $x$, Pesin theory associates to
$E^-_f(x)$ and $E^+_f(x)$ local stable and unstable manifolds tangent
to these subspaces at $x$.

The \emph{Pesin stable manifold} of $x$ is defined as
\[
\mathcal{W}^-_f(x)=\left\{y\in M : \limsup_{n\to \infty}\frac{\log d(f^n(x),f^n(y))}{n}<0\right\},
\]
and the \emph{Pesin unstable manifold} of $x$ as
\[
\mathcal{W}^+_f(x)=\left\{y\in M : \limsup_{n\to \infty}\frac{\log d(f^{-n}(x),f^{-n}(y))}{n}<0\right\}.
\]

\begin{thm}[Pesin Stable Manifold Theorem {\cite{Pes}}]\label{pesin}
Let $M$ be a compact Riemannian manifold and let $f:M\to M$ be a $C^2$
diffeomorphism preserving an invariant measure $\mu$. Then for
$\mu$-almost every regular point $x$, if
$
k=\dim E^-_f(x),$ 
there exists a $C^{1+\beta}$ embedded $k$-dimensional disk
$\mathcal{W}^-_{f,\mathrm{loc}}(x)$ centered at $x$ such that
$
T_x\mathcal{W}^-_{f,\mathrm{loc}}(x)=E^-_f(x),
$
and
$
\mathcal{W}^-_{f,\mathrm{loc}}(x)\subset \mathcal{W}^-_f(x).
$
Similarly, there is a local unstable manifold
$\mathcal{W}^+_{f,\mathrm{loc}}(x)$ tangent to $E^+_f(x)$ at $x$.
\end{thm}


We also recall the absolute continuity properties of Pesin stable and unstable manifolds, which we shall use to prove ergodicity of some elements of the centralizer.

\begin{thm}[Absolute continuity of Pesin stable/unstable manifolds
{\cite{Pes}}]\label{pesabcont}
Let $f$ be a $C^2$ volume-preserving diffeomorphism of a compact
Riemannian manifold $M$. Then the family of local stable manifolds
$\{\mathcal W^-_{f,\mathrm{loc}}(x)\}$ is absolutely continuous when restricted to a set where the dimension of
$E^-_f$ is constant.
More precisely, if $S_1$ and $S_2$ are sufficiently small smooth disks
transverse to these local stable manifolds inside a foliated box, and
if $h:D_1\subset S_1\to D_2\subset S_2$ denotes the stable holonomy map,
then $h$ is absolutely continuous with respect to the induced
Riemannian measures on $S_1$ and $S_2$. The analogous statement holds
for the local unstable manifolds.
\end{thm}

In particular, if a measurable set has zero measure in $M$, then its intersection
with almost every local Pesin stable leaf has zero measure with respect to the Riemannian volume restricted to the leaves; conversely, if
a measurable set has zero leafwise measure on almost every local leaf,
then it has zero measure in $M$.

\section{The center-fixing centralizer}\label{sec:czf=zf}

Let $f$ be a perturbation of the (twisted) Weyl chamber flow. Denote
\[
\CZ(f)=\{g\in \mathcal Z(f):\; g(x)\in \W^c_f(x)\ \text{for all }x\in X\},
\]
the center-fixing part of the centralizer of $f$.

In this section, we show that $\mathcal Z(f)/\CZ(f)$ injects into the discrete group $\mathrm{Out}(\Gamma)$, following the method of
\cite{Witte}.

In the case when $G$ is the semisimple homogeneous space, we use the following result from \cite{Wang}.

\begin{thm}[Theorem 4.1, Corollary 4.3 of \cite{Wang}, see also Theorem 1.1 of \cite{Witte}]
    Let $X=G/\Gamma$ be a semisimple homogeneous space. Let $f_0$ be a non-trivial element of the Weyl chamber flow, and let $f\in \mathrm{Diff}^\infty(X)$ be a sufficiently small $C^1$-perturbation of $f_0$. Then we have
$$|\mathcal{Z}^r(f)/\CZ^r(f)|<C,$$
where $C$ is a constant depending only on $G$.
\end{thm}

Now we focus on the twisted case.

We record a description of the outer automorphism group of $\Gamma$.

\begin{lem}\label{lem:finout}
The outer automorphism group of $\Gamma=\Gamma_0\ltimes_\rho\Z^N$ is
commensurate to the centralizer of $\rho(\Gamma_0)$ in $\mathrm{GL}_N(\Z)$.  More precisely,
\[
\mathrm{Out}(\Gamma)\doteq Z_{\mathrm{GL}_N(\Z)}(\rho(\Gamma_0))
\doteq Z_{\mathrm{GL}_N(\R)}(\rho(G_0))\cap\mathrm{GL}_N(\Z).
\]
In particular, every element of $\mathrm{Out}(\Gamma)$ may be
represented (up to inner automorphism and a fixed power) by the action of
$\operatorname{id}\ltimes A$ on
$(G_0\ltimes_\rho\R^N)/(\Gamma_0\ltimes_\rho\Z^N)$ for some
$A\in Z_{\mathrm{GL}_N(\Z)}(\rho(\Gamma_0))$.
\end{lem}

We postpone the proof to the end of the section.

\begin{prop}\label{prop:coset-fix} (See Theorem 1.1 of \cite{Witte})
Let $X=G
/\Gamma$ be a (twisted) semisimple homogeneous space. Let $G^c=Z_G^0(f_0)$ be the identity component centralizer (in $G$) of $f_0$.  Then
any homeomorphism $X\to X$ that preserves the $G^c$--cosets can be
written as a product of an element in $\mathrm{Out}(\Gamma)$ and a homeomorphism
that fixes every $G^c$--coset.
\end{prop}

\begin{proof}
Let $\varphi:X\to X$ be a homeomorphism preserving the $G^c$--cosets.
Lift $\varphi$ to a homeomorphism $\tilde\varphi:G\to G$ that
preserves the left cosets of $G^c$ and satisfies
$d_{G}(\tilde\varphi(\mathrm{id},0),(\mathrm{id},0))\le R(X)$, where
$R(X)$ is the diameter of $X$.

The lift induces a group automorphism
$\tilde\varphi_*:\Gamma\to\Gamma$ by
$(\gamma,v)\mapsto \tilde\varphi(\mathrm{id},0)^{-1}\tilde\varphi(\gamma,v)$.
By Lemma \ref{lem:finout} (and reduction by inner
automorphisms), we may compose $\tilde\varphi$ with an automorphism
$(\operatorname{id}\ltimes A)$, $A\in
Z_{\mathrm{GL}_N(\Z)}(\rho(\Gamma_0))$, so that the resulting automorphism
$\psi=(\operatorname{id}\ltimes A)\circ\tilde\varphi$ acts trivially on
$\Gamma$.

We claim that $\psi$ fixes every $G^c$--coset.  By construction $\psi$
is a bounded homeomorphism of the (noncompact) group $G$. If for
some $x\in G$ the cosets $G^c\psi(x)$ and $G^c x$ are distinct, then
they would be at a bounded distance (because $\psi$ is bounded), contradicting
the fact that distinct cosets of $G^c$ in $G$ are not at bounded
distance. Hence $\psi(x)\in G^c x$ for all $x$, i.e. $\psi$ fixes the
$G^c$--cosets. This proves the proposition.
\end{proof}

As a corollary we obtain the desired relation between the centralizer and
outer automorphisms.

\begin{cor}
Let $f$ be a perturbation of the element $f_0$ of the twisted Weyl
chamber flow. Then the quotient
$\mathcal Z(f)/\CZ(f)$ injects into $\mathrm{Out}(\Gamma)$. Equivalently,
$\mathcal Z(f)/\CZ(f)$ is isomorphic to a subgroup of
$\mathrm{Out}(\Gamma)$.
\end{cor}

\begin{proof} There is a natural map $\mathcal Z(f ) \to \mathrm{Out}(\Gamma)$ given by the action of lifts
of elements of $\mathcal Z(f)$ on the lattice $\Gamma$. Let $g$ lie in the kernel of the map.

Since $g\in \mathcal Z(f)$, it preserves the center, stable, and unstable foliations of $f$. Let $\phi$ be the leaf conjugacy from $(f_0,\W^c_{f_0})$ to $(f,\W^c_f)$. Therefore, the
lifting of $\phi^{-1} g\phi$ preserves the center leaves of $f_0$, which are
the $G^c$-cosets in $G$. By Proposition \ref{prop:coset-fix} the map
$\phi^{-1} g\phi$ is a product of an outer automorphism of $\Gamma$ and a homeomorphism that fixes the  $G^c$--coset. The leaf conjugacy $\phi$ is $C^0$-close to identity, so we may choose a lifting close to identity so that the induced action of $\phi$ on $\Gamma$ is trivial. Since the induced action of $g$ on $\Gamma$ is an inner automorphism, we see that $\phi^{-1} g\phi$ induces a trivial action on $\mathrm{Out}(\Gamma)$. This shows that $\phi^{-1} g\phi$ fixes the center foliations of $f_0$. So $g$ fixes the center leaves of $f$ and $g\in\CZ(f)$. Thus the kernel is exactly
$\CZ(f)$, and the quotient embeds into $\mathrm{Out}(\Gamma)$.
\end{proof}

Now we go back to computing the structure of the outer automorphism group.
\begin{proof}[Proof of Lemma \ref{lem:finout}]
Let $\psi$ be an automorphism of $\Gamma$. Projecting to the
$\Gamma_0$--factor gives an induced automorphism
$\psi_G:\Gamma\to\Gamma_0$, $\psi_G(\gamma)=\pi(\psi(\gamma,0))$, where
$\pi:\Gamma\to\Gamma_0$ is the projection. Since $\Gamma_0$ is a
cocompact lattice in the $\R$-split semisimple group $G_0$, Margulis
superrigidity implies that $\psi_G$
extends to an automorphism of $G_0$ preserving $\Gamma_0$. Since
$\mathrm{Out}(G_0)$ is finite, after composing $\psi$ with an inner
automorphism by an element of $\Gamma_0\ltimes\{0\}$, we may assume
$\psi_G=\operatorname{id}$ up to a power.

Thus we have $\psi(\gamma,0)=(\gamma,v_\gamma)$ for some $v_\gamma\in\Z^N$ and
all $\gamma\in\Gamma_0$. From the group law in $\Gamma$ we obtain
\[
v_{\gamma\gamma'}=v_\gamma+\rho(\gamma)v_{\gamma'}\qquad(\gamma,\gamma'\in\Gamma_0),
\]
so $\gamma\mapsto v_\gamma$ is a $\Z^N$-valued cocycle for the
$\Gamma_0$-action by $\rho$. Since $\Gamma_0$ is a lattice in a higher‑rank semisimple Lie group $G_0$ and $\rho$ does not have a trivial subrepresentation, classical results (by Raghunathan~\cite{Raghunathan}; or Zimmer's cocycle superrigidity~\cite{Zimmer}) imply that the cocycle is cohomologous to the trivial cocycle, i.e. conjugating $\psi$ by a suitable element of
$\{\mathrm{id}\}\ltimes\R^N$ (an inner automorphism) eliminates this
cocycle. Hence (up to inner automorphism) we may assume $v_\gamma=0$
for all $\gamma$.

Now restrict $\psi$ to the normal subgroup $\{\mathrm{id}\}\ltimes\Z^N$.
This gives a homomorphism $\Z^N\to\Z^N$, so
$\psi(\mathrm{id},v)=(\mathrm{id},Av)$ for some
$A\in\mathrm{GL}_N(\Z)$. Using
$\psi(\gamma,v)=\psi(\gamma,0)\psi(\mathrm{id},\rho(\gamma)^{-1}v)=\psi(\mathrm{id},v)\psi(\gamma,0)$ and the previous
normalization so that $v_\gamma=0$, we obtain
$\rho(\gamma)A\rho(\gamma)^{-1}=A$ for all $\gamma\in\Gamma_0$. Thus
$A\in Z_{\mathrm{GL}_N(\Z)}(\rho(\Gamma_0))$.

Therefore every outer automorphism class is represented (after composing
with an inner automorphism and up to a power) by $\operatorname{id}\ltimes A$ with
$A\in Z_{\mathrm{GL}_N(\Z)}(\rho(\Gamma_0))$, and
$Z_{\mathrm{GL}_N(\Z)}(\rho(\Gamma_0))=Z_{\mathrm{GL}_N(\R)}(\rho(G_0))\cap\mathrm{GL}_N(\Z)$
by Zariski density of $\rho(\Gamma_0)$ in $\rho(G_0)$. This completes the proof.
\end{proof}

\section{Dynamics in the centralizer}\label{sec:cla}
In this section, we establish some dynamical properties of elements of the centralizer. We denote by $\CZ_0(f)$ the identity component of $\CZ(f)$, and we shall focus on elements in this subgroup of the centralizer.

\subsection{Elements of the centralizer are $C^0$ close to affine}
Following \cite{Wang}, we estimate the difference between the $su$-holonomies
of $f$ and $f_0$ to show that the elements of $\CZ_0(f)$ are $C^0$-close to affine
translations along center leaves of $f_0$ in $X$.

Let $f$ be a $C^1$–small perturbation of an element of the (twisted) Weyl chamber flow $f_0$ in $G/\Gamma$.

Let $\phi$ be the leaf conjugacy from $(X,\mathcal W^c_{f_0})$ to
$(X,\mathcal W^c_f)$ provided by Theorem \ref{leafconj}. For
$g\in\CZ_0(f)$ denote $\hat g:=\phi^{-1}\,g\,\phi$. Since $g$ fixes the
center leaves of $f$, $\hat g$ fixes the center leaves of $f_0$, which are
left cosets of $G^c$.

Denote by $\tilde g$ a lift of $\hat g$ to the cover $G$ of $G/\Gamma$,
and by $\widetilde{\mathcal W}^c_{f_0}$ the lift of the center leaves of
$f_0$ to $G$ (these are the left cosets of $G^c$).

We first obtain a uniform estimate on the behavior of $\hat g$ at every
$x\in X$.

\begin{prop}\label{gtransl}
For any $0<\epsilon<1$ and $R>0$ there exists $\delta>0$ such that for any
$f\in\mathrm{Diff}^r(X)$ with $d_{C^1}(f,f_0)<\delta$, and for any
$x\in X$ and $g\in\CZ_0(f)$ satisfying
$d_{\widetilde{\mathcal W}^c_{f_0}}(\tilde g(\tilde x),\tilde x)\le R$, we have
\[
d_{\widetilde{\mathcal W}^c_{f_0}}(\tilde g(\tilde y),g_1(\tilde y))
\le \epsilon\; d_{\widetilde{\mathcal W}^c_{f_0}}(\tilde{g}(\tilde x),\tilde x),
\;\; \forall y\in X.
\]

Here $g_1$ denotes the left translation in $G/\Gamma$ determined by
$\tilde g(\tilde x)\tilde x^{-1}$, i.e.
$g_1(y)=\big(\tilde g(\tilde x)\tilde x^{-1}\big)\cdot y$, and for $z,w$ in the same
$G^c$–coset
$d_{\widetilde{\mathcal W}^c_{f_0}}(z,w):=d_{G^c}(z w^{-1},\mathrm{id})$.

\end{prop}

Since $f$ is partially hyperbolic and dynamically coherent, any
$g\in\mathcal Z(f)$ satisfies $g(\mathcal W^*_f(x))=\mathcal W^*_f(g(x))$
for $*=s,u,c,cs,cu$. If moreover $g\in\CZ_0(f)$ then $g$ fixes every center
leaf of $f$, hence $g$ commutes with the global $su$–holonomies between
center leaves. Consequently the behavior of $g$ at different points is
controlled by the $su$–holonomies of $f$. The proof below follows the
proof of Proposition 5.7 in \cite{Wang}.

\begin{proof}

The stable, unstable and center foliations of $f_0$ lift to left cosets
of subgroups of $G$, hence the stable/unstable holonomies of $f_0$ are
affine: for $x,y\in X$, any $su$-path $\gamma_0$ connecting $x$ and $y$, and any $z\in\mathcal W^c_{f_0}(x)=G^c\cdot x$, the global holonomy in the universal cover $G$ is given by
\[
h^{f_0}_{\gamma_0}(\tilde z)=\tilde z \tilde x^{-1}\cdot  \tilde y.
\]

Let $g\in\CZ_0(f)$. Let $\phi(\gamma)$ be an $f$-su path connecting $\phi(x)$ and $\phi(y)$ in $X$. Let $\gamma_0$ be an $f_0$-su path connecting $x$ and $y$.

Since $g$ commutes with the global stable and
unstable holonomies of $f$, we have
$\tilde g(\tilde y)=\widehat h^f_{\gamma} (\tilde g( \tilde x))$ and
$g_1(\tilde y)=\tilde{g}(\tilde x)\tilde x^{-1}\cdot \tilde y=h^{f_0}_{\gamma_0}(\tilde g(\tilde x))$, where
$\widehat h^f_{\gamma}:=\phi^{-1}\,h^f_{\phi^{-1}(\gamma)}\,\phi$
is the conjugated holonomy. To compare $\tilde g(\tilde y)$ and $g_1(\tilde y)$, it suffices to compare the holonomies of $f$ and $f_0$.

By Proposition \ref{faccess}, there exist constants $C,N>0$ such
that for every $x,y\in X$ there is an $su$–path $\phi(\gamma)$ from $\phi(x)$ to $\phi(y)$ of length
$\le C$ with at most $N$ legs. (We note that $X$ is compact so $d_X(x,y)$ is uniformly bounded).

Fix a curve $c$ in the center leaf
joining $\tilde x$ to $z=\tilde g(\tilde x)$ with length $l(c)<2d_{\tilde \W^c_{f_0}}(\tilde x,\tilde z)$. For any
point $x'\in c$ the $f$–$su$–path from $\phi(x')$ to $\phi \hat h^f_\gamma(x^\prime)$ also has
length $\le C_1$ where $C_1>C$ is a constant depending on $R>d_{\tilde \W^c_{f_0}}(\tilde g(\tilde x),\tilde x)$, the diameter of $X$ and the dynamics of $f_0$.

On the other hand, we notice that under the right-invariant coordinates of $G$, $D(h^{f_0}_{\gamma_0})=\mathrm{Id}$ since $h^{f_0}_{\gamma_0}$ is simply a right translation in the cover $G$. Applying Theorem \ref{c1hol},  we have $\|D\hat h^f_\gamma-\mathrm{Id}\|\to 0$  as $d_{C^1}(f,f_0)\to 0$. Since the leaf conjugacy $\phi$ has uniform $C^1$-norm along the center foliation, this shows that
\[
\begin{aligned}
d(\widehat h^f_{\gamma}(z),h^{f_0}_{\gamma_0}(z))
&\le \max_{x'\in c}\|D(\widehat h^f_{\gamma})|_{x'}
- D(h^{f_0}_{\gamma})|_{x'}\|\; l(c)\\
&= \max_{x'\in c}\|D(\widehat h^f_{\gamma})|_{x'}
- \mathrm{Id}\|\; l(c)\\
&\le \epsilon\, d_{\tilde \W^c_{f_0}}(\tilde x,z),
\end{aligned}
\]
which shows $d_{\widetilde{\mathcal W}^c_{f_0}}(\tilde g(\tilde y),g_1( \tilde y))
\le \epsilon\; d_{\widetilde{\mathcal W}^c_{f_0}}(\tilde g(\tilde x),\tilde x)$. 
\end{proof}

Motivated by Proposition \ref{gtransl} we make the following definition.

\begin{defi}\label{def:projepsclose}
We say that $g\in \CZ_0(f)$ is projectively $\epsilon$-close to a left
translation $g_0$ if for every $y\in X$ the lifts satisfy
\[
d_{\widetilde{\mathcal W}^c_{f_0}}(\tilde g(y),g_0(y))
\le \epsilon\; d_{\widetilde{\mathcal W}^c_{f_0}}(g_0(y),y),
\]
where $\tilde g$ is the lift of $\hat g=\phi^{-1} g\phi$ to $G$ and
$\phi$ is the leaf conjugacy between $f_0$ and $f$. We note that $g_0$ is a left translation, so the right-hand side $d_{\widetilde{\mathcal W}^c_{f_0}}(g_0(y),y)$ is a constant not depending on $y$.
\end{defi}

\subsection{Centralizer acts as a Lie group}

First we recall a slight variation of Theorem 3.1 in \cite{Wang}, which
implies that $\CZ(f)$ is a Lie group acting freely, properly, and smoothly
on the lifted center leaf $\widetilde{\mathcal W}^c_F$.

\begin{thm}[Theorem 3.1, \cite{Wang}]\label{czlieg}
Let $M$ be a compact Riemannian manifold and let $F$ be a partially
hyperbolic diffeomorphism of $M$. Assume $F$ is accessible and dynamically
coherent, center $r$‑bunched with narrow band spectrum, and that a lift of
$F$ has global holonomy on the lift of a center leaf
$\widetilde{\mathcal W}^c_F(x_0)$. Then the $\mathcal W^c$‑fixing,
$C^\infty$ centralizer $\CZ(F)$ is a Lie group acting freely and
properly on $\widetilde{\mathcal W}^c_F(x_0)$. Moreover the action
\[
\CZ(F)\times \widetilde{\mathcal W}^c_F(x_0)\to\widetilde{\mathcal W}^c_F(x_0),
\qquad (g,x)\mapsto g(x),
\]
is jointly $C^\infty$.
\end{thm}

Applying this to the affine diffeomorphism $f_0$, we see that $\CZ_0(f_0)$ acts freely and properly on $\W^c_{f_0}(x_0)\simeq G^c$. This implies that $\CZ_0(f_0)=G^c$ and so the centralizer $\mathcal Z(f_0)$ is virtually the algebraic centralizer of $f_0$.

In fact, we can also prove the following stronger statement about the centralizer of affine maps on $X$. 

\begin{prop}\label{prop:affinecent}
    Let $f_1$ be a left translation by an element of $G^c$, which is partially hyperbolic and accessible. Then for any $g\in \mathrm{Homeo}(X)$, if $g\circ f_1=f_1\circ g$, and $g$ fixes the center foliation of $f_0$, and $g_*=\mathrm{id} \in \mathrm{Out}(\Gamma)$, then $g$ is a left translation on $X$ that commutes with $f_0$.
\end{prop}
\begin{proof}
 Since $g_*=\mathrm{id} \in \mathrm{Out}(\Gamma)$, we take $\tilde g$ to be a lifting of $g$, such that $\tilde g(y\gamma)=\tilde g(y)\gamma=\tilde g(y)y^{-1}\cdot y\gamma$ for any $\gamma\in \Gamma$.
 
Since $g$ fixes the center foliations, consider the continuous map $b(y)=\tilde g(y)y^{-1}\in G^c$, where $G^c$ is the identity component of the center of $f_0$ in $G$. This is a well defined function on $X$. Since $g$ commutes with $f_1$, we have that $\tilde g(f_1(\tilde x))=f_1\tilde g(\tilde x)=\tilde g(\tilde x)\tilde x^{-1} f_1(\tilde x)$, i.e. $b$ is $f_1$-invariant. Since $f_1$ is a partially hyperbolic accessible affine map in a compact homogeneous space, classical results in \cite{PughShubStarkov} show that $f_1$ is ergodic. Therefore, the map $b$ is constant.

    This implies that $g$ is a left translation by an element of $G^c$.
\end{proof}

By the observations made in Section \ref{prel:ph}, Propositions \ref{fparhyp} and \ref{faccess} (normal hyperbolicity, smoothness of center
leaves, accessibility, and existence of global holonomy on the lift), the
hypotheses of Theorem \ref{czlieg} are satisfied for sufficiently small $C^1$
perturbations $f$ of $f_0$ and the
corresponding lifted center leaf in $G$. Hence $\CZ_0(f)$ is a Lie
group acting freely and properly on $\widetilde{\mathcal W}^c_f$. By
Theorem 1 of \cite{BM}, the
action $\CZ_0(f)\times\widetilde{\mathcal W}^c_f(x)\to\widetilde{\mathcal W}^c_f(x)$
is jointly $C^\infty$.

\subsection{Linear growth of center elements}

Suppose $g\in\CZ_0(f)$ lies in the center of $\CZ_0(f)$. We claim that
$g$ is projectively $\epsilon$-close to a translation $g_0\in\CZ_0(f_0)\simeq G^c$
which lies in the center of $G^c$. 

If $f_0$ is a generic element of the twisted Weyl chamber flow, then $G^c$ is abelian, and any element $g\in \CZ_0(f)$ in the connected component of the identity is $C^0$-close to a translation.

\begin{prop}\label{prop:gtog_0generic}
Fix $R>0$, suppose $f$ is a sufficiently
$C^1$–small perturbation of a generic element of the twisted Weyl chamber flow $f_0$. If $\dim \mathcal Z(f)\ge k-\min k_i+2$, where $k=\mathrm{rank}\; G_0$ and $k_i$ are the ranks of simple components of $G_0$, then there exists $g\in\CZ_0(f)$, such that $\inf_{x\in X}d_{\tilde\W^c_f}(g(x),x)\le R$, and a left translation $g_0$, such that

\begin{enumerate}
\item $g_0$ is not projectively $0.1$-close to $f_0$ when projected to any simple component of $\mathfrak h$;

\item $g$ is projectively $\epsilon$-close  to $g_0$ (see Definition \ref{def:projepsclose});

\item $$e^{-\epsilon^2 |n|}\|v\|\le \|Dg^n(v)\|\le e^{\epsilon ^2|n|}\|v\|$$ for any $n\in \Z$ and any $v\in E^c_f$.
\end{enumerate}
\end{prop}

The proof of
(2) is a direct application of Proposition \ref{gtransl}. The proof of (1) is exactly the same as Proposition 5.6 of \cite{Wang}, and (3) is exactly the same as the proof of Proposition 6.4 of \cite{Wang}. 

Now we treat the case when $f_0$ is non-generic. In what follows, we denote by $Z(\CZ_0(f))$ the center of the identity component of $\CZ(f)$.

\begin{lem}\label{centerisom}
Fix $0<\epsilon<0.1$. Let $f_0$ be a non‑trivial element of the (twisted) Weyl chamber flow on
$X=G/\Gamma$, and let $f\in\mathrm{Diff}^\infty(X)$ be a sufficiently
small $C^1$–perturbation of $f_0$, suppose $\mathcal Z(f)\doteq \mathcal Z(f_0)$. Then for any $g\in Z(\CZ_0(f))$, such that $  d_{\tilde \W^c_{f_0}}(\tilde g(\tilde x),\tilde x)\ge 1$ for some $x\in X$, there exists $g_0\in Z(\CZ_0(f_0))$, such that $g$ is projectively $\epsilon$-close to $g_0$. 
\end{lem}

\begin{proof}
Since $\mathcal Z(f)/\CZ_0(f)$ is discrete, the virtual isomorphism
\(\mathcal Z(f)\doteq\mathcal Z(f_0)\) identifies the connected
center-fixing part of \(\mathcal Z(f)\) with that of the algebraic
centralizer.  The latter has center \(Z(\CZ_0(f_0))\simeq \mathbb R^\ell\).
Thus \(Z(\CZ_0(f))\simeq \mathbb R^\ell\).

Now suppose $1\le d_{\tilde\W^c_{f_0}}(\tilde g(\tilde x),\tilde x)\le 3$. By Proposition \ref{gtransl}, for any $\epsilon_1>0$, if $f$ is sufficiently $C^1$-close to $f_0$, then for any $y\in X$, $$d_{\tilde \W^c_{f_0}}(\tilde g(y),(\tilde g(\tilde x)\tilde x^{-1})\cdot y)\le \epsilon_1 .$$ 

Similarly, for any other element $h\in \CZ_0(f)$, if $d_{\tilde \W^c_{f_0}}(\tilde h(\tilde x),\tilde x)\le 1$, we have $$d_{\tilde \W^c_{f_0}}(\tilde h(y),(\tilde h(\tilde x)\tilde x^{-1})\cdot y)\le \epsilon_1.$$  

Plugging in $y=\tilde h(\tilde x)$ in the first equation and $y=\tilde g(\tilde x)$ in the second, and using $\tilde g\tilde h=\tilde h\tilde g$, this yields
\begin{equation}\label{commutbd}
    d_{\tilde \W^c_{f_0}}\big((\tilde g(\tilde x)\tilde x^{-1})\cdot (\tilde h(\tilde x)\tilde x^{-1}),(\tilde h(\tilde x)\tilde x^{-1})\cdot (
\tilde g(\tilde x)\tilde x^{-1}) \big)\le 2\epsilon_1.
\end{equation}

On the other hand, since the action of $\CZ_0(f)$ on $\tilde \W^c_{f_0}$ is smooth, free and proper (since the leaf conjugacy is smooth along the center), and $\tilde\W^c_{f_0}$ is the same dimension as $\CZ_0(f)$, the action is locally transitive and the element $\tilde h(\tilde x)\tilde x^{-1},h\in \CZ_0(f)$ ranges over any element $h_0\in G^c\le G$, $d(h_0,\mathrm{id})\le 1$. Therefore, let $g_1=\tilde g(\tilde x)\tilde x^{-1}\in G$, then we have $$d_{G^c}(Ad(g_1)h_0,h_0)\le 4\epsilon_1 $$ for any $h_0\in \CZ_0(f_0)$, $d(h_0,\mathrm{id})\le 1$. Therefore, when $\epsilon_1$ is sufficiently small, there exists a group element $g_0 \in Z(\CZ_0(f_0))$ such that $$d_{G^c}(g_0,(\tilde g(\tilde x)\tilde x^{-1}))\le\frac{\epsilon}{2}.$$ By abuse of notation, we denote the left translation $g_0(x)=g_0\cdot x$, and pick $\epsilon_1$ even smaller if necessary, such that $10\epsilon_1\le \frac{\epsilon}{3}$, then since $d_{G^c}(g_0,id)\ge 1-\frac{\epsilon}{2}$, we have $$d_{\tilde \W^c_{f_0}}(\tilde g(y), g_0(y))\le 10\epsilon_1+\frac{\epsilon}{2}\le  \epsilon(1-\frac{\epsilon}{2}) \le \epsilon\, d(g_0(y),y)$$ for any $y\in X$, i.e. $g$ is projectively $\epsilon$-close to $g_0\in Z(\CZ_0(f_0))$.

Suppose $g\in Z(\CZ_0(f))$ and $d_{\tilde\W^c_{f_0}}(\tilde g(\tilde x),\tilde x)\ge 3$. Since $Z(\CZ_0(f))\simeq \R^l$, for any $k\in \N$, there exists a well-defined map $g^{\frac{1}{2^k}}\in Z(\CZ_0(f))$. Since $g^{\frac{1}{2^k}}\to \mathrm{id}$ as $k\to \infty$, there exists $k\in \N$, such that $d_{\tilde\W^c_{f_0}}(\tilde g^{\frac{1}{2^k}}(\tilde x),\tilde x)\le 1$ and $d_{\tilde\W^c_{f_0}}(\tilde g^{\frac{1}{2^{k-1}}}(\tilde x),\tilde x)> 1$. On the other hand, we have $d_{\tilde\W^c_{f_0}}(\tilde g^{\frac{1}{2^{k-1}}}(\tilde x),\tilde x)\le d_{\tilde\W^c_{f_0}}(\tilde g^{\frac{1}{2^{k}}}(\tilde x),\tilde x)+d_{\tilde\W^c_{f_0}}(\tilde g^{\frac{1}{2^{k}}}(g^{\frac{1}{2^{k}}}(\tilde x)),g^{\frac{1}{2^{k}}}(\tilde x))\le (2+\epsilon)d_{\tilde\W^c_{f_0}}(\tilde g^{\frac{1}{2^k}}(\tilde x),\tilde x)\le 3$. Therefore, by our analysis above, $g^{\frac{1}{2^{k-1}}}$ is projectively $\epsilon$-close to some $g_0'\in Z(\CZ_0(f_0))$. Let $g_0=(g_0')^{2^{k-1}}$, then we have $$d(\tilde g(y),\tilde g_0(y))\le 2^{k-1}\cdot \epsilon\, d( g_0'(y),y)=\epsilon \,d(g_0(y),y).$$ This shows that $g$ is also projectively $\epsilon$-close to $g_0\in Z(\CZ_0(f_0)).$


\end{proof}

In addition, we next show that we can pick $g\in Z(\CZ_0(f))$ to satisfy the condition of the lemma, so that $\langle f,g\rangle$ is $C^0$-close to $\langle f_0,g_0\rangle$. Furthermore, these elements exhibit at most linear growth along the center. 
\begin{prop}\label{choiceg}
For $0<\epsilon<1$, there exists $\delta>0$ such that if $f_0,f\in\mathrm{Diff}^\infty(X)$ satisfy the conditions of Theorem \ref{cenrignongen:twisted} and $d_{C^1}(f,f_0)<\delta$, then there exist $g\in Z(\CZ_0(f))$ and $g_0\in Z(\CZ_0(f_0))$, such that \begin{enumerate}
    \item $g_0$ is not in the $0.1$-cone of $f_0$ in any simple component of $G_0$;

    \item for any $k,l\in \Z$, the composition $f^kg^l$ is projectively $\epsilon$-close to $f_0^kg_0^l$;

    \item for any $k,l\in \Z$, the composition $g_1=f^kg^l$ has bounded  growth in the center: for any $x\in X$, $y\in\mathcal W^c_f(x,loc)$ and any $n\in \Z$,
\[
 (1-\epsilon)d_{\tilde{\W^c_f}}(x,y)\le d_{\tilde{\W^c_f}}(g_1^n(x),g_1^n(y))\le (1+\epsilon)d_{\tilde{\W^c_f}}(x,y).
\]
\end{enumerate}

\end{prop}

\begin{proof}
 (1) Fix any $0<\epsilon_1<\epsilon$. We first pick an appropriate choice of $g$. We identify $Z(\CZ_0(f))\simeq  Z(\CZ_0(f_0))\simeq \mathbb R^\ell$. 
 Fix a base point $x_0\in \tilde \W^c_{f_0}(x_0)\simeq G^c$. We consider the map $$\psi: Z(\CZ_0(f))\to \R^\ell,\quad \psi(g)= c(\tilde g(x_0)x_0^{-1})$$ where $c: G^c\to \R^\ell $ is the projection to the center. If $\delta$ is sufficiently small, then by repeatedly applying Proposition \ref{gtransl} and using that the center is abelian, we have for any $k\in \N$ and $g\in \CZ_0(f)$: \begin{equation}\label{equ: psibd}
     |\psi(g^k)-k\psi(g)|\le \epsilon_1 k|\psi(g)|.
 \end{equation} By Lemma \ref{centerisom}, if $d(\tilde g^k(x_0),x_0)\ge1$, then $g^k$ is projectively $\epsilon_1$-close to $\psi(g^k)$, so $g^k$ is $2\epsilon_1$-close to $k\psi(g)$. 
 
\begin{lem}
    The differential $d\psi$ is non-degenerate at the identity point.
\end{lem} 

\begin{proof}
    If not, there exists $u\in \mathfrak g(Z(\CZ_0(f))),\|u\|=1$, $d\psi(u)=0$. Therefore, we have $|\psi(\exp(\frac{u}{t}))|\le \|D^2\sigma(\cdot,x_0)\|\cdot |\frac{u}{t}|^2$ as $t\to \infty$, where $\sigma$ is the smooth, proper and free action of $\CZ_0(f)$ on $\tilde \W^c_{f_0}(x_0)\simeq G^c$. Equation \ref{equ: psibd} then implies that \begin{equation}\label{equ:2nddegbd}
    |\psi(\exp(tu))|\le 2t^2|\psi(\exp(\frac{u}{t}))|\le 2t^2\|D^2\sigma(\cdot,x_0)\|\cdot |\frac{u}{t}|^2=2\|D^2\sigma(\cdot,x_0)\|.
\end{equation}

On the other hand, by properness of the action $\sigma$, there also exists a constant $K$, such that $d(\exp_{\CZ_0(f)}(tv)(x_0),x_0)\ge 2$ for any $v\in \mathfrak g(\CZ_0(f)),|v|=1, t >K$. By Lemma \ref{centerisom}, this shows that  $\exp_{\CZ_0(f)}(tv)$ is projectively close to a translation by element in $Z(\CZ_0(f_0))$, which implies that $|\psi(\exp(tu))|\ge 1$ when $t>K$. Therefore, our earlier analysis shows that $\psi(\exp(Ks\cdot u))$ is $2\epsilon_1$-close to $s\psi(Ku)$. In particular, we have $$|\psi(\exp (Ks\cdot u))|\ge \frac{s}{2}$$ for any $s\in \N$, which contradicts Equation \ref{equ:2nddegbd}.
\end{proof}

By the lemma above, $\psi$ is locally surjective around the identity, so we may pick $h\in Z(\CZ_0(f))$ such that $\psi(h)$ is not in the $0.1$-cone of $f_0$ in any simple component of $G_0$. Then by applying enough power so that $d(\tilde h^k(x_0),x_0)>1$, we have $h^k$ is projectively $2\epsilon_1$-close to $k\cdot \psi(h)$, which is also not in the $0.1$-cone of $f_0$. Therefore, take $g=h^k$, $g_0=k\cdot \psi(h)$, then we have $g_0$ is not in the $0.1$-cone of $f_0$ and $g$ is projectively $2\epsilon_1$-close to $g_0$.

 (2) Now we claim that $g^kf^l$ is projectively $\epsilon$-close to $g_0^kf_0^l$. Indeed, repeatedly applying Proposition \ref{gtransl} gives $$d_{\tilde\W^c_{f_0}}(\tilde g^k\tilde f^l(y),g_0^kf_0^l(y))\le 2(|k|+|l|)\epsilon_1.$$ On the other hand, since $g_0$ is not in the $0.1$-cone of $f_0$ and they are both elements in $Z(\CZ_0(f_0))$, we have $C(|k|+|l|)\ge d_{G^c}(g_0^kf_0^l,\mathrm{id})\ge \frac{|k|+|l|}{C}$ for a constant $C$ that only depends on fixed metric in $G^c$. This shows that $$d_{\tilde\W^c_{f_0}}(\tilde g^k\tilde f^l(y),g_0^kf_0^l(y))\le 2(|k|+|l|)\epsilon_1\le \epsilon \,d_{G^c}(g_0^kf_0^l,\mathrm{id}),$$ when $\epsilon_1<\frac{\epsilon}{2C}$. This proves (2).

For (3), by the virtual Lie-group isomorphism assumption, 
\[
\dim \CZ_0(f)=\dim \CZ_0(f_0)=\dim \W^c_{f_0}=\dim \W^c_f.
\]
Since the action of \(\CZ_0(f)\) on the lifted center leaf is free, proper and
smooth, its orbits are open and closed in the center leaf. For any point $y\in \W^c_f(x,loc)$, there exists $h\in \CZ_0(f)$ such that $y=h(x)$. Therefore,  $$d_{\tilde{\W^c_f}}(g^n_1(x),g^n_1(y))=d_{\tilde{\W^c_f}}(g^n_1(x),g^n_1(h(x)))=d_{\tilde{\W^c_f}}(h(g^n_1(x)),g^n_1(x)).$$
 Applying Proposition \ref{gtransl} to $h$, and using that the leaf conjugacy is $C^1$-close to identity along the center, we have  $$|d_{\tilde{\W^c_f}}(g^n_1(x),g^n_1(y))-d_{\tilde{\W^c_f}}(h(x),x)|\le \epsilon\, d_{\tilde{\W^c_f}}(h(x),x)=\epsilon\, d_{\tilde{\W^c_f}}(y,x).$$ This proves (3).
\end{proof}

\section{Hyperbolicity of topological leaves}\label{topleaf}

In this section we establish some topological partially hyperbolic
properties of certain elements in the centralizer and find an appropriate
subgroup of $\CZ_0(f)$ which can be compared with a higher‑rank restriction
of the (twisted) Weyl chamber flow. Here we treat the case when the center of $\mathcal Z(f_0)$ has dimension at least two in each simple component: we build topological Lyapunov foliations and show that elements of the centralizer act hyperbolically in them in a weak H\"older sense.

Let $f_0$ be an element of the (twisted) Weyl chamber flow such that the center of $\mathcal Z(f_0)$ has dimension at least two in each simple component of $G_0$, and let $f$ be a $C^1$–perturbation of $f_0$. We shall
show that $\CZ_0(f)$ preserves a canonical family of topological
foliations, and that many elements of $\CZ_0(f)$ contract and expand these
foliations in the same way that elements of $\CZ_0(f_0)$ expand and
contract the coarse Lyapunov foliations. For now we do not assume any condition on the dimension of $\CZ_0(f)$.

A restriction of the (twisted) Weyl chamber flow in $G/\Gamma$ is said to
be \emph{genuinely higher-rank} if its lifting to $G$ has no algebraic
rank‑1 factor.

We say that an action $\alpha$ is a \emph{topological perturbation} of an
affine action $\alpha_0$ if it preserves a family of ``topological
Lyapunov foliations'' and acts on them dynamically like the action
$\alpha_0$.

\begin{defi}\label{topper}
Let $\alpha_0: A \to \mathrm{Diff}(X)$ be a restriction of the (twisted)
Weyl chamber flow, where $A$ is an abelian group. Let $\{\mathcal W^i_{\alpha_0}\}$ be the coarse Lyapunov
foliations of $\alpha_0$, i.e.\ the maximal intersections of stable and
unstable foliations of elements of $\alpha_0$, with Lyapunov functionals $\chi_i$; and let $\mathcal W^c_{\alpha_0}$
be the center foliation of $\alpha_0$.

We say that a continuous action $\alpha:A \to \mathrm{Homeo}(M)$ is an
$\epsilon,r$–\emph{topological perturbation} of $\alpha_0$ if there exists
a family of continuous foliations $\mathcal W^i_\alpha$ and a center
foliation $\mathcal W^c_\alpha$ that are locally transverse,
preserved by the whole action $\alpha$, and satisfying the following properties:
\begin{itemize}
\item the foliations $\mathcal W^i_\alpha$ and $\mathcal W^c_\alpha$ have
  the same dimensions as $\mathcal W^i_{\alpha_0}$ and $\mathcal W^c_{\alpha_0}$;
  
  \item $\W^c_\alpha$ is uniformly $C^0$–$\epsilon$–close to $\mathcal W^c_{\alpha_0}$;

  \item holonomies of paths of length $\le \frac{1}{\epsilon}$ of $\W^i_\alpha$ and  $\mathcal W^i_{\alpha_0}$ are $C^0$-$\epsilon$-close;

   \item for any $a\in A$, $\alpha(a)$ is projectively $\epsilon/2$-close to $\alpha_0(a)$;

\item for any $a\in A$, the map $\alpha(a)$ contracts or expands
  $\mathcal W^i_\alpha$ exponentially with the same signature as
  $\alpha_0(a)$, whenever $\alpha_0(a)$ is not in the $\epsilon$-cone of $\ker \chi_i$. More precisely, for any $x\in X$ and
  $y\in\mathcal W^i_\alpha(x)$ there exists $N$ (depending on $x,y$)
  such that for all $|n|>N$, $\mathrm{sign}(n)=-\mathrm{sign} (\chi_i(a))$
  \[
    d(\alpha(na)x,\alpha(na)y)\le e^{r\chi_i(a) n}.
  \]

\item for any $a\in A$, the map $\alpha(a)$ has very small expansion or contraction in $\W^c_\alpha$. I.e., for any $x\in X$,
  $y\in\mathcal W^c_\alpha(x)$ and $n\in \Z$,
  \[
    e^{-\epsilon^2\|\alpha_0(a)\||n|}d(x,y)\le d(\alpha(na)x,\alpha(na)y)\le e^{\epsilon^2\|\alpha_0(a)\||n|}d(x,y).
  \] 

  Here we use $\|\alpha_0(a)\|$ to denote the norm of the translation of $\alpha_0(a)$ in $ \mathfrak h$.

\item For each pair $\chi_i,\chi_j\in \Lambda_{\alpha_0}$ with $\chi_i\not\in \R\chi_j$, there exists $a\in A$, such that $\chi_i(a)<0, \chi_j(a)<0$, and $\alpha_0(a)$ is not in the $\epsilon$-cone of the walls $\ker \chi_i$ or $\ker \chi_j$.
\end{itemize}
\end{defi}

In this section we show that under our assumptions $f$ belongs to a
topological perturbation of a genuinely higher‑rank restriction of the
(twisted) Weyl chamber flow. 

\begin{prop}\label{topph}
 Let $f_0$ be an element of the (twisted) Weyl chamber flow, such that the center of $\mathcal Z(f_0)$ has dimension at least two in each simple component of $G_0$, and let $f$ be a $C^1$–perturbation of $f_0$ as in Theorems \ref{cenrigtwgen} or \ref{cenrignongen:twisted}. 

For any $\epsilon>0$, there exists $0<\bar\epsilon<\epsilon$, and a sub-action $\alpha:\Z^2\to\mathrm{Diff}(X)$ of the action by $\CZ_0(f)$ which is an $(\bar\epsilon,r)$–topological
perturbation of $\alpha_0:\Z^2\to\mathrm{Diff}(X)$, where $\alpha_0$ is a
genuinely higher‑rank restriction of the (twisted) Weyl chamber flow, and $r$ is a constant that only depends on $f_0$.
\end{prop}

We postpone the proof to Section \ref{sec:pftopph}.

\subsection{Canonical choice of the topological foliations}\label{cantopfol}
Suppose $f$ is a $C^1$–perturbation of an element $f_0$ of the twisted
Weyl chamber flow, and assume the center of $\mathcal Z(f_0)$ has dimension at least two in each simple component of $G_0$.

We construct the following topological foliations, which are preserved
by elements of $\CZ(f)$. Later (Section \ref{sec:contrfoli}) we show
that they are uniformly expanded or contracted by certain elements in
the centralizer $\CZ(f)$.

For a nonzero root or weight $\mu$ of the (twisted) Weyl chamber flow with
$\mu(f_0)\neq 0$, denote by $\mathcal W^{\mu,c}_{f_0}$ the foliation
whose leaves are cosets of the group
\[
G^{\mu,c}=\langle G^c,\exp(V_\mu)\rangle
\]
generated by the center subgroup $G^c$ of $f_0$ and the Lyapunov
foliation $G^\mu$ of $f_0$. (Remark: the dimension of the leaves of
$\mathcal W^{\mu,c}_{f_0}$ may exceed the sum of the dimensions of
$G^c$ and $V_\mu$, since the subgroups need not be jointly
integrable.)

Define $\mathcal W^\mu_{f_0}$ to be the foliation whose leaves are the
intersection of leaves of $\mathcal W^{\mu,c}_{f_0}$ with the stable
leaves of $f_0$ if $\mu(f_0)<0$, or with the unstable leaves of $f_0$ if
$\mu(f_0)>0$.

Let $\phi$ be the leaf conjugacy between $(f,\mathcal W^c_f)$ and
$(f_0,\mathcal W^c_{f_0})$. Denote the pulled–back foliations by
$\mathcal W^{\mu,c}_\# := \phi(\mathcal W^{\mu,c}_{f_0})$. Define
the topological Lyapunov foliations of $f$ by
\begin{equation}\label{equ:Wmuf}
\mathcal W^\mu_\# \;=\; \mathcal W^{\mu,c}_\# \;\cap\; \mathcal W^*_f,
\end{equation}
where $*=s$ if $\mu(f_0)<0$ and $*=u$ if $\mu(f_0)>0$. By construction,
leaves of $\mathcal W^\mu_\#$ are $C^0$–close to those of
$\mathcal W^\mu_{f_0}$; hence $\mathcal W^\mu_\#$ is a topological
subfoliation of $\mathcal W^s_f$ or $\mathcal W^u_f$ accordingly.

Note that if $\nu(f_0)=0$ for some root or weight $\nu$, then leaves of
$\mathcal W^\mu_\#$ may have larger dimension than the $\mu$–eigenspace.
Moreover, it is possible that distinct $\mu\neq\mu'\in\Lambda_0$ satisfy a
containment relation, e.g. $\mathcal W^\mu_\#$ may be a subfoliation of
$\mathcal W^{\mu'}_\#$.

For any $g\in\CZ(f)$, the map $g$ preserves the topological foliation
$\mathcal W^\mu_\#$ for every root or weight $\mu$. Indeed $g$ fixes the
center leaves, so it preserves $\mathcal W^{\mu,c}_\#$, and $g$ also
preserves the stable and unstable foliations $\mathcal W^s_f$ and
$\mathcal W^u_f$.

We shall prove in Section \ref{sec:contrfoli} that if
$g\in\CZ_0(f)$ is sufficiently $C^0$–close to a homogeneous diffeomorphism
$g_0\in\CZ_0(f_0)$, then $g$ uniformly expands or contracts the foliations
$\mathcal W^\mu_\#$ whenever $g_0$ uniformly expands or contracts the
corresponding foliations $\mathcal W^\mu_{f_0}$.

\subsection{Exponential contraction of topological Lyapunov foliations}\label{sec:contrfoli}
We now prove that elements $g \in \CZ_0(f)$ act by contraction and expansion on the topological leaves $\W^\mu_\#$ as expected, provided that the corresponding $g_0 \in \CZ_0(f_0)$ exhibits similar behavior. We say that $g_0$ is not in the $\epsilon_0$-cone of the Weyl chamber wall $\ker(\mu)$ if $g_0$ is not projectively $\epsilon_0$-close to any $g_1 \in \CZ_0(f_0)$ with $\mu(g_1) = 0$.

The following lemma is an analogue of Lemma 6.3 of  \cite{Wang}, and establishes exponential expansion and contraction in the topological coarse Lyapunov foliations, the idea of the proof is also similar, except that we need to add an assumption that $\mathrm{sign}(\mu(g_0) + \nu_0(g_0)) = \mathrm{sign}(\mu(g_0))$ for any root or weight $\nu_0\in \Lambda_0$ with $\nu_0(f_0) = 0$.

Note that we will apply the lemma to $g_0\in Z(\CZ_0(f_0))$, so this assumption is always satisfied.

\begin{lem}\label{gcontr}
    Suppose $g \in \CZ_0(f)$ is projectively $\epsilon_0/2$-close to an element $g_0 \in \CZ_0(f_0) \simeq G^c$. For a root or weight $\mu\in \Lambda_0$ of the (twisted) Weyl chamber flow, suppose $\mathrm{sign}(\mu(g_0) + \nu_0(g_0)) = \mathrm{sign}(\mu(g_0))$ for any root or weight $\nu_0\in \Lambda_0$ with $\nu_0(f_0) = 0$, and suppose $g_0$ is not in the $\epsilon_0$-cone of the Weyl chamber wall $\ker(\mu)$. Then there exist $\epsilon > 0$, and $N > 0$, depending only on $g_0$, such that for every $x \in X$ and $y \in \W^\mu_\#(x)$ with $d(x,y) < \epsilon$, and for any $n > N$, we have:
    \[
    d(g^n(x), g^n(y)) \le e^{n \theta \mu(g_0)/4} d(x,y)^{\theta^2}, \quad \text{if } \mu(g_0) < 0,
    \]
    \[
    d(g^{-n}(x), g^{-n}(y)) \le e^{-n \theta \mu(g_0)/4} d(x,y)^{\theta^2}, \quad \text{if } \mu(g_0) > 0.
    \]
    Here, $\theta < 1$ is the bi-Hölder constant of the leaf conjugacy $\phi$.
\end{lem}

\begin{proof}[Proof of Lemma \ref{gcontr}]

    We restrict our discussion to the case $\mu(g_0) < 0$, where $g_0$ uniformly contracts $\W^\mu_{f_0}$. The case $\mu(g_0) > 0$ follows by considering $g^{-1}$.

By the definition of projective closeness, for every \(z\in X\), the
center displacement
\[
\tilde g(z)z^{-1}\in G^c
\]
has Cartan projection lying in the prescribed \(\epsilon_0/2\)-cone about
\(g_0\). Hence, for each \(j\ge0\), the one-step displacement
\[
\tilde g(\tilde g^j z)(\tilde g^j z)^{-1}
\]
satisfies the same cone condition. Applying the corresponding one-step
estimate successively along the orbit \(z,\tilde g z,\ldots,\tilde g^{n-1}z\),
we obtain for any $w \in \W^\mu_{f_0}(z)$,
\[
d(\tilde g^n z,\tilde g^n w)
   \le e^{n\mu(g_0)/2}d(z,w),
\qquad w\in \W^\mu_{f_0}(z).
\]

    Furthermore, since the leaves $\W^c_{f_0}$ and $\W^\mu_{f_0}$ are cosets of subgroups of $G$ with orthogonal Lie algebras, there exists a constant $\epsilon > 0$ such that for any $z \in X$ and $w \in \W^\mu_{f_0}(z)$ with $d(z,w) < \epsilon$, we have:
    \[
    d(z, \tilde{\W}^c_{f_0}(w)) \le 2d(z,w).
    \]
    Since $f$ is a $C^1$-perturbation of $f_0$ and $\W^\mu_\#$ is a subfoliation of $\W^s_f$ or $\W^u_f$, a similar property holds for $f$: for any $z \in X$ and $w \in \W^\mu_\#(z)$ with $d(z,w) < \epsilon$, we have $d(z, \tilde{\W}^c_f(w)) \le 2d(z,w)$ in the cover $G$.

    Therefore, using the Hölder continuity of the leaf conjugacy $\phi$, for every $x \in X$ and every $y \in \W^\mu_\#$ with $d(x,y) \le \epsilon$, define:
    \[
    y' = \phi\big(\W^\mu_{f_0}(\phi^{-1}(x))\big) \cap \W^c_f(y), \quad y'_n = \phi(\W^\mu_{f_0}(\phi^{-1}(g^n(x))) \cap \W^c_f(g^n(y)).
    \]
    Then we have:
    \[
    \begin{aligned}
    d(g^n(x), g^n(y)) &\le 2 d(g^n(x), y'_n) \\
    &\le 2C d(\hat{g}^n(\phi^{-1}(x)), \phi^{-1}(y'_n))^\theta \\
    &\le 2C \left(e^{n\mu(g_0)/2} d(\phi^{-1}(x), \phi^{-1}(y'))\right)^\theta \\
    &\le 4C e^{n \theta \mu(g_0)/2} d(\phi^{-1}(x), \phi^{-1}(y))^\theta \\
    &\le 4C^2 e^{n \theta \mu(g_0)/2} d(x,y)^{\theta^2}.
    \end{aligned}
    \]
    By choosing $N$ sufficiently large, we obtain the desired result.
\end{proof}

\subsection{Proof of Proposition \ref{topph}}\label{sec:pftopph}

We now prove Proposition \ref{topph}, which establishes the existence
of a topological perturbation of a higher‑rank action inside the
centralizer.

In the case of Theorem \ref{cenrigtwgen}, the proof is exactly the same as the proof of Proposition 6.1 of \cite{Wang}. Indeed, Proposition \ref{prop:gtog_0generic} establishes small exponential growth along the center, and Lemma \ref{gcontr} shows exponential contraction and expansion in the topological stable and unstable foliations that is comparable to the growth of $\langle f_0,g_0\rangle$. The subtlety here is that, a priori, it is possible that one of the Weyl chambers of $\langle f_0,g_0\rangle$ has cone angle smaller than $\epsilon^2$, and the last property in Definition \ref{topper} will not hold. We avoid this situation using the same proof as in Section 6.4 of \cite{Wang}. When $\bar\epsilon$ is sufficiently small, we may pick $\alpha_0$ to be a small perturbation of $\langle f_0,g_0\rangle$ that collapses the small Weyl chambers, so that the Weyl chamber cones of $\alpha_0$ all have non-zero angles much larger than $\bar{\epsilon}$. 

We now prove Proposition \ref{topph} in the case of Theorem \ref{cenrignongen:twisted}.
\begin{proof}[Proof of Proposition \ref{topph}]
In the case of Theorem \ref{cenrignongen:twisted}, Proposition
\ref{choiceg} yields a subgroup of $Z(\CZ_0(f_0))$ that contains a
genuinely higher‑rank action $\alpha_0$ (which includes $f_0$). Let
$\alpha$ be the corresponding action in
$\CZ_0(f)$, from Proposition
\ref{choiceg}. We pick $\bar\epsilon$ sufficiently small so there are elements not close to the Weyl chamber walls in each cone of $\alpha_0$. By Proposition \ref{choiceg}, any element of
$\alpha=\langle f,g\rangle$ within a bounded set is projectively $\bar\epsilon$-close to the
corresponding element in $\alpha_0=\langle f_0,g_0\rangle$. 

Let $\{\mathcal W^i_{\alpha_0}\}$ denote the coarse Lyapunov foliations of
$\alpha_0$. Define $\mathcal W^i_\alpha$ to be the pull-back of the joint integration of $\W^i_{\alpha_0}$ with $\W^c_{f_0}$, intersected with $\W^s_f$ or $\W^u_f$, similar to how we defined $\mathcal W^\mu_\#$ in Section
\ref{cantopfol}. By construction, they are also the joint integration of the
topological foliations $\mathcal W^\mu_\#$, if $\mathcal W^i_{\alpha_0}$ is the joint integration of the
corresponding algebraic $\mathcal W^\mu_{f_0}$. The center foliations $\W^c_\alpha$ and $\W^c_{\alpha_0}$ are uniformly $C^0$-close because they are just center foliations of $f$ and $f_0$. By construction, the holonomies of  $\mathcal W^i_\alpha$ are also $C^0$–close to $\mathcal W^i_{\alpha_0}$, since they are just part of stable or unstable holonomies of $\W^*_{f_0}$ and $\W^*_{f}$ for $*=s$ or $u$.
Lemma \ref{gcontr} shows that elements of $\alpha$ exhibit the
expected exponential contraction or expansion along $\mathcal W^i_\alpha$
whenever the corresponding algebraic element in $\alpha_0$ does so.

Moreover, Proposition \ref{choiceg} implies that elements of $\alpha$
have at most linear growth along the center foliation
$\mathcal W^c_f=\mathcal W^c_\alpha$. Combining these facts shows that
$\alpha$ is an $\bar\epsilon$–topological perturbation of the genuinely
higher‑rank action $\alpha_0$ in the sense of Definition \ref{topper}.

Finally, accessibility of the family $\{\mathcal W^i_\#\}$ holds because
by construction $\mathcal W^s_f$ is the joint integration of $\{\mathcal W^i_\#:\chi_i(f)<0\}$ and $\mathcal W^u_f$
is the joint integration of $\{\mathcal W^i_\#:\chi_i(f)>0\}$; since $f$ itself is accessible, these
topological Lyapunov foliations together are accessible as well. This completes
the proof.
\end{proof}

\section{Rigidity}\label{sec:cocrig}

In this section, we use a geometric argument to prove Theorems
\ref{cenrigtwgen} and \ref{cenrignongen:twisted}. We select an appropriate
element $g\in\CZ_0(f)$ and use the partially hyperbolic properties of $g$
established earlier to show that the action generated by $f$ and $g$ is
smoothly conjugate to a restriction of the twisted Weyl chamber flow.

The idea, following \cite{DK}, \cite{Vinhage}, and \cite{VinWang}, is to
show that a ``large part'' of the group of closed loops of Lyapunov
foliations of $f$ are generated by stable cycles of the higher‑rank
action $\alpha=\langle f,g\rangle:\Z^2\to\mathrm{Diff}(X)$. This implies
that the periodic cycle functional of any cocycle of $\alpha$ is
trivial, which in turn shows that the cocycle given by center
translations of $\alpha$ is cohomologous to a constant. This yields a
conjugacy between $\alpha$ and a restriction of the (twisted) Weyl
chamber flow. Then the classical normal‑form argument  from \cite{KS} upgrades the conjugacy
to a smooth conjugacy.

\subsection{Cocycle formulation}

Let $\alpha_1:A\to\mathrm{Diff}(X)$ be a group action and let $S$ be a
(group) target. Recall that a cocycle $\beta:A\times X\to S$ satisfies
\[
\beta(ab,x)=\beta(a,\alpha_1(b)(x))\,\beta(b,x),\qquad\forall a,b\in A,\ x\in X.
\]
A cocycle is cohomologous to a constant if there exists a homomorphism
$s:A\to S$ and a transfer map $H:X\to S$ with
\[
\beta(a,x)=H(\alpha_1(a)(x))\,s(a)\,H(x)^{-1},\qquad\forall a\in A,\ x\in X.
\]

We consider the action
\[
\alpha=\langle f,g\rangle:\Z^2\to\mathrm{Diff}(X).
\]

\begin{itemize}
    \item If $f_0$ is a generic element of the (twisted) Weyl chamber flow, choose $g$ to be projectively $\epsilon$-close to $g_0$ that is not projectively $0.1$-close to $f_0$ in any simple component of $\mathfrak{g}(G_0)$.

    \item For the non-generic case, choose
  $g\in\CZ_0(f)$ in the center of $\CZ_0(f)$ with $g$ being $C^0$–close to
  $g_0\in Z(\CZ_0(f_0))$, where \(g_0\) is not in the \(0.1\)-cone of \(f_0\) in any simple component, as given in Proposition \ref{choiceg}.
\end{itemize}

Let $\alpha_0=\langle f_0,g_0\rangle:\Z^2\to\mathrm{Diff}(X)$. By Proposition \ref{topph} we have shown that $\alpha$ is a
topological perturbation of $\alpha_0$ on a generating set $A_0$. That is, for $a\in A_0$ the
map $\alpha(a)$ exhibits the same type of contraction/expansion as
$\alpha_0(a)$ along the corresponding topological stable, unstable and
center foliations. Moreover, our choices ensure that $\alpha_0$ is a
genuinely higher‑rank restriction of the (twisted) Weyl chamber flow.

Define the conjugated action $\hat\alpha$ by the leaf conjugacy:
\begin{equation}\label{hatalpha}
\hat\alpha(a)=\phi^{-1}\,\alpha(a)\,\phi,
\end{equation}
where $\phi$ is the leaf conjugacy from $(f_0,\mathcal W^c_{f_0})$ to
$(f,\mathcal W^c_{f})$.

Define $\hat\beta:\Z^2\times X\to G^c$ by
\begin{equation}\label{cocbeta}
\hat\beta(a,x)=\widetilde{\hat\alpha(a)}(\tilde x)\,\tilde x^{-1},
\end{equation}
where $\widetilde{\hat\alpha(a)}$ denotes the lift of the homeomorphism
$\hat\alpha(a)$ to $G$ that fixes the center leaves of $f_0$, and
$\tilde x$ is a chosen lift of $x$. By construction $\hat\beta$ is a
Hölder cocycle over $\hat\alpha$.

Using contraction along stable and unstable foliations (see Lemma
\ref{gcontr}) and the algebraic structure
of the Lyapunov cycle group, we obtain the following cocycle rigidity
result.

\begin{thm}\label{cocrig}
Let $\hat\alpha$ be the $\Z^2$–action defined in
\eqref{hatalpha}. Then any Hölder cocycle
$\beta:\Z^2\times X\to G^c$ over $\hat\alpha$ which is $C^0$–close
to a constant on a generating set is cohomologous to a constant via a
H\"older transfer map $H$.
\end{thm}

We now deduce conjugacy to the homogeneous model by applying Theorem
\ref{cocrig} to $\hat\beta$ defined in \eqref{cocbeta}. The argument
parallels the proof of Theorem 1.1 in \cite{DK}.

\begin{prop}\label{prop:fholconj}
Under the assumptions of Theorems \ref{cenrigtwgen} and 
\ref{cenrignongen:twisted}, the diffeomorphism
$f$ is Hölder conjugate to an element of the (twisted) Weyl chamber flow, and H\"older conjugacy identifies $\CZ_0(f)$ with a subgroup of $\CZ_0(f_0)$.
\end{prop}

\begin{proof}
Take $\hat\beta$ as in \eqref{cocbeta}. We first verify that for $f$
satisfying the hypotheses of the stated theorems the cocycle $\hat\beta$
is well‑defined and $C^0$–close to a constant on a generating set.

If $f_0$ is as in Theorems \ref{cenrigtwgen} or \ref{cenrignongen:twisted} and $f$ is a sufficiently
small $C^1$–perturbation satisfying the hypotheses, then Proposition
\ref{topph} implies $\alpha$ is a topological perturbation of $\alpha_0$
and hence $\hat\beta$ is well‑defined and $C^0$–close to a constant.

Thus Theorem \ref{cocrig} applies: $\hat\beta$ is cohomologous to a
constant via a Hölder transfer map $H:X\to G^c$.

Define $h_1:X\to X$ by $h_1(x)=H(x)^{-1}\cdot x$. Then for any
$\hat g=\hat\alpha(a)$ with $a\in\Z^2$ we have
\begin{equation}\label{chomg}
h_1\circ\hat g(x)=s(a)\,\alpha_0(a)\cdot h_1(x),\qquad\forall x\in X,
\end{equation}
where $s:\Z^2\to G^c$ is the constant homomorphism produced by
the cohomology.

We claim $h_1$ is a homeomorphism. Lifting $h_1$ to the cover $G$, since $h_1$ is a semiconjugacy between $\tilde g$ and $s(a)\alpha_0(a)$, it sends stable/unstable foliations of $\hat \alpha$ to those of $s\circ \alpha_0$. By construction, we also know  $h_1$ fixes the center leaves of $\W^c_{\alpha_0}$. Therefore, if $h_1(x)=h_1(y)$ then for any
stable holonomy $h^s$ between center foliations, we
have $h_1(h^s(x))=h_1(h^s(y))$. The same is true for unstable holonomies. Hence the orbit of $x$ under the $su$-holonomy
group in $\W^c_{\hat \alpha}$ is contained in $h_1^{-1}(\{h_1(x)\})$, which is bounded. If $x\neq y$, pick an $su$-path from $x$ to $y$, then the orbit of the holonomy $\{h_\gamma^n(x)\}_n\subset h_1^{-1}(\{h_1(x)\})$ has a converging subsequence. But by
Proposition \ref{hologroup}, the holonomy group acts freely in a generic
center leaf of $G$, so $\{h_\gamma^n(x)\}_n$ cannot be bounded. This gives a contradiction. Thus, we must have $x=y$ and $h_1$ is injective; since
it is continuous in the closed manifold $X$ and homotopic to the identity, $h_1$ is a
homeomorphism.

It follows that $\hat f$ is Hölder conjugate to an
element in $G^c$. Conjugating back via $\phi$ shows
that $f$ is Hölder conjugate to a left translation of an element of $G^c$. If $f_0$ is a generic element of the twisted Weyl chamber flow, then $G^c=D\times \{0\}$ and this shows that $f$ is conjugate to an element of the (twisted) Weyl chamber flow.

If $f_0$ is a non-generic element of the (twisted) Weyl chamber flow, this shows that $f$ is H\"older conjugate to a partially hyperbolic, accessible left translation $f_1$ on $X$. The conjugacy takes any element $g\in \CZ_0(f)$ to a homeomorphism that commutes with $f_1$. By Proposition \ref{prop:affinecent}, this shows that the conjugacy maps $\CZ_0(f)$ to a subgroup of $\CZ(f_1)\subset G^c$. On the other hand, under the assumption of Theorem \ref{cenrignongen:twisted}, we have $\CZ_0(f)\doteq G^c$ as a group. This shows that $\CZ(f_1)=\CZ(f_0)$, which forces $f_1$ to be an element of the (twisted) Weyl chamber flow. 

In both cases, this shows that the conjugacy identifies $\CZ_0(f)$ with a subgroup of $\CZ_0(f_0)$. In the setting of Theorem \ref{cenrignongen:twisted}, this actually shows that conjugacy identifies $\CZ_0(f)$ with $\CZ_0(f_0)$.
\end{proof}

Since $\CZ_0(f)$ is H\"older conjugate to a substantial part of $\CZ_0(f_0)$, we may show that generic elements of $\CZ_0(f)$ are
partially hyperbolic; hence the coarse Lyapunov foliations of $\alpha$
are $C^\infty$–smooth. The proof is very similar to the proof of Theorem 7.11 of \cite{Wang}, and we postpone the proof to Section \ref{alphaph}.

By standard normal‑form theory for algebraic
manifolds, one can then upgrade the H\"older conjugacy to a smooth conjugacy. The
smoothness of $h_1$ along the coarse Lyapunov foliations follows as in
Step 4 of Section 2.2.3 in \cite{KS}. Global smoothness follows from
Theorem 2.1 of \cite{KSpat}.

\subsection{Lyapunov cycles in general}

We begin with an analysis of the group of Lyapunov cycles for an action
$\alpha$ which preserves a family of topological Lyapunov foliations.

Let $\alpha$ be an action on a manifold $M$ preserving a family of
topological Lyapunov foliations $\{\mathcal W^i_\alpha\}$. A path in
$M$, each of whose legs is contained in one of these foliations is called a
\emph{Lyapunov path} of $\alpha$. We denote such a path by its sequence
of endpoints $[x_0,x_1,\dots,x_k]$, where $x_i\in\mathcal W^{j}_\alpha(x_{i-1})$
for some index $j$ and for all $1\le i\le k$. Here
$\mathcal W^j_\alpha:=\bigcap_{a:\chi_j(a)<0}\mathcal W^s_{\alpha(a)}$
is the Lyapunov foliation of $\alpha$ associated to the index $j$. Let
$P(\alpha)$ denote the set of Lyapunov paths for $\alpha$ of finite
length (finitely many legs).

A Lyapunov path is an $\alpha$–\emph{cycle} if it starts and ends at the
same point. Fix $x_0\in M$ and denote by $C_{x_0}(\alpha)$ the set of
$\alpha$–cycles based at $x_0$. The set $C_{x_0}(\alpha)$ is a group
under concatenation of paths; it carries the coarsest topology making
each endpoint map continuous (see Section 7.2 of \cite{VinWang}). If
the Lyapunov foliations are accessible, then the cycle groups based at
different points are conjugate, so one may omit the basepoint and write
$C(\alpha)$ for the Lyapunov cycle group. Denote by $C_0(\alpha)$ the
subgroup of $C(\alpha)$ consisting of cycles whose lifts to the
universal cover of $M$ remain closed.

An $\alpha$–cycle in $C(\alpha)$ is called \emph{stable} if there exists
$g\in\operatorname{Im}(\alpha)$ such that the entire path is contained
in the stable manifold $\mathcal W^s_g(x_0)$ (i.e.\ the joint integration
of those $\mathcal W^i_\alpha$ with $\chi_i(g)<0$). Let $S(\alpha)$ be
the closure (in the topology of $C(\alpha)$) of the normal subgroup of
$C(\alpha)$ generated by stable cycles. Note that stable cycles are
contractible: iterating by the contracting element $g$ shrinks the
cycle into an arbitrarily small ball, hence $S(\alpha)\subset C_0(\alpha)$.

\begin{defi}[Lyapunov cycle functional]
Given a H\"older cocycle $\beta$ of an action $\alpha$ with topological
Lyapunov foliations, the \emph{Lyapunov cycle functional}
$P_\beta:C_{x_0}(\alpha)\to G^c$ is defined by
\[
P_\beta\big([x_0,x_1,\dots,x_k=x_0]\big)
=\prod_{i=0}^{k-1} p_\beta(x_i,x_{i+1}),
\]
where for $y\in\mathcal W^i_\alpha(x)$ and any $a\in A$ with
$\chi_i(\alpha(a))<0$,
\[
p_\beta(x,y)=\lim_{n\to\infty}\beta(na,x)^{-1}\beta(na,y)
\]
(the limit exists by contraction along the foliation and the H\"older
property of $\beta$).
\end{defi}

The following proposition characterizes cocycles cohomologous to
constants in terms of the Lyapunov cycle functional.

\begin{prop}\label{consttrivialfunctional}
Let $\alpha:A\to\mathrm{Diff}(X)$ be an action with accessible
topological Lyapunov foliations. A H\"older cocycle
$\beta:A\times X\to G^c$ is cohomologous to a constant via a
H\"older transfer map $H:X\to G^c$ if and only if the
Lyapunov cycle functional $P_\beta:C(\alpha)\to G^c$ is trivial.
\end{prop}

Furthermore, since stable cycles are always contractible, we have the following.
\begin{lem}\label{lem:stableinkernel}
For any H\"older cocycle $\beta$ of a group action $\alpha$ we have
$S(\alpha)\subset\ker P_\beta$.
\end{lem}

Consequently the Lyapunov cycle functional factors through the quotient
$C(\alpha)/S(\alpha)$.

\subsection{Lyapunov cycles of $\alpha_0$ and $\hat\alpha$}

\subsubsection{Lifting to the universal central extension}
Let $\tilde G$ denote the universal cover of $G$ if $G$ is semisimple,
or the universal central extension of $G$ (as defined in Section
\ref{cenext}) if $G$ is the semidirect product $G_0\ltimes_\rho\R^N$.
The lattice $\Gamma\subset G$ naturally extends to a lattice
$\tilde\Gamma\subset\tilde G$. Let $\Theta:\tilde G\to G$ be the
covering map (in the universal‑cover case) or the factoring map (in the
central‑extension case). This induces a map, also denoted
$\Theta:\tilde G/\tilde\Gamma\to G/\Gamma$.

We first lift the homogeneous action $\alpha_0$ to $\tilde G$. Fix a
generating set $\{f_0,g_0\}$ of $\alpha_0$. Choose lifts $\tilde f_0$ and
$\tilde g_0$ of $f_0$ and $g_0$ to $\tilde G$. Note that $\tilde f_0$
and $\tilde g_0$ need not commute in $\tilde G$, but since $f_0$ and
$g_0$ commute in $G$ we have
$\tilde f_0\tilde g_0=\tilde g_0\tilde f_0\cdot c$ for some
$c\in Z(\tilde G)$. This defines an action
$\tilde\alpha_0:F_2\to\tilde G\subset\mathrm{Diff}(\tilde G/\tilde\Gamma)$.
The stable and unstable foliations of $\tilde\alpha_0(a)$ for
$a\in F_2$ are preserved by each other, so we may define the Lyapunov
foliations of $\tilde\alpha_0$ as their maximal intersections. By
construction $\Theta$ maps Lyapunov foliations of $\tilde\alpha_0$ to
those of $\alpha_0$.

Next we lift the topological perturbation $\hat\alpha$ to
$\tilde G/\tilde\Gamma$. There is a natural embedding of Lie algebras
$\mathfrak g(G)\hookrightarrow\mathfrak g(\tilde G)$. For
$a=(1,0)$ or $(0,1)$ and $\tilde x\in\tilde G/\tilde\Gamma$ with
$x=\Theta(\tilde x)$, define
\[
\tilde\alpha(a,\tilde x)
=\exp_{\tilde G}\Big(\log_{G}\big(\hat\alpha(a,x)\,\alpha_0(a,x)^{-1}\big)\Big)\,
\tilde\alpha_0(a,\tilde x),
\]
where $\log_G$ is well defined because $\hat\alpha(a,x)\,\alpha_0(a,x)^{-1}$
is $C^0$–close to the identity. This prescription extends to an action
$\tilde\alpha:F_2\to\mathrm{Diff}(\tilde G/\tilde\Gamma)$.

Under the assumption that either $f_0$ has no zero weight or
$\rho$ has no symplectic component, the centralizer
$\CZ_0(f_0)=G^c$ embeds into $\tilde G$ as a subgroup. A direct
calculation (see Equations (12) and (13) of \cite{VinWang}) shows that
the topological stable, unstable and center foliations of $\hat\alpha$
lift uniquely to topological stable, unstable and center foliations of
$\tilde\alpha$. Since these maps commute when projected to $X$, the
elements of $\tilde\alpha$ preserve the lifted foliations. Therefore the
topological Lyapunov foliations and the center foliation of $\hat\alpha$
lift naturally to those of $\tilde\alpha$ in $\tilde G/\tilde\Gamma$.

Finally, the cocycle $\beta$ lifts to
$\tilde\beta:F_2\times\tilde G/\tilde\Gamma\to G^c$ by
\[
\tilde\beta(a,\tilde x)=\beta(a,\Theta(\tilde x)),
\]
and becomes a cocycle over $\tilde\alpha$.

\subsection{Stable cycles ``generate'' contractible cycles}

We first show that $C(\tilde\alpha)/S(\tilde\alpha)$ is small for the
action $\tilde\alpha$ by comparing it with the corresponding set for
$\tilde\alpha_0$.

It is proved in Section 7 of \cite{VinWang} that the image of the
$su$–functional on contractible cycles of $\tilde\alpha_0$ is trivial.

\begin{prop}[Theorem 7.2, \cite{VinWang}]\label{minper}
Let $\alpha_0:\Z^2\to\mathrm{Diff}(X)$ be a genuinely higher-rank restriction of the
(twisted) Weyl chamber flow in $X=G/\Gamma$. If
$G=G_0\ltimes_\rho\R^N$ is a semidirect product, assume further that
either $\rho$ has no symplectic component or $\alpha_0$ has no zero
weight in $\R^N$. Let $\tilde\alpha_0: F_2\to \mathrm{Diff}(\tilde G/\tilde \Gamma)$ be the lift of $\alpha_0$ to the
universal central extension $\tilde G$. Then
$C^0(\tilde\alpha_0)/S(\tilde\alpha_0)$ is minimally almost periodic;
that is, any continuous homomorphism from
$C^0(\tilde\alpha_0)/S(\tilde\alpha_0)$ to a Lie group is trivial.
\end{prop}

We obtain a similar result for $\tilde\alpha$–cycles. Define a natural
map $P$ from topological Lyapunov paths of $\hat\alpha$ to Lyapunov
paths of $\alpha_0$ as follows. Given an $\tilde\alpha$-Lyapunov path
with endpoints $[x_0,x_1,\dots,x_k]$, let $[y_0,y_1,\dots,y_k]$ be the
(unique) $\tilde\alpha_0$–Lyapunov path with $y_0=x_0$, $y_k\in
\mathcal W^c_{\tilde\alpha_0}(x_k)$, and for each $1\le i\le k$ we have
$y_i\in\mathcal W^j_{\tilde\alpha_0}(y_{i-1})$ whenever
$x_i\in\mathcal W^j_{\tilde\alpha}(x_{i-1})$. This defines $P$. Define
an inverse map $Q$ from Lyapunov paths of $\tilde\alpha_0$ to Lyapunov
paths of $\tilde\alpha$ similarly.

A priori a Lyapunov cycle of $\tilde\alpha$ need not map to a cycle of
$\tilde\alpha_0$. However every stable cycle of $\tilde\alpha$ maps to
a stable cycle of $\tilde\alpha_0$. Using Proposition \ref{minper} we
conclude that contractible $\tilde\alpha_0$–cycles map to contractible
$\tilde\alpha$–cycles. The inverse direction is subtler: for a
uniformly bounded $\tilde\alpha_0$–path $\gamma$ whose endpoints have
center distance $1<d_{\mathcal W^c_{\tilde\alpha_0}}(\gamma(0),\gamma(1))<2$,
since $\tilde\alpha$ is a topological perturbation of $\tilde\alpha_0$,
$Q(\gamma)$ cannot be a closed cycle; thus images of endpoints of
$\tilde\alpha_0$–cycles mapped to $\tilde\alpha$ must be discrete.
However there is a homotopy from the identity cycle to any
$\tilde\alpha$–cycle in the relevant homotopy class, so the discrete
image must be trivial. The complete argument is given in Section 12 of
\cite{VinWang}.

\begin{prop}[Theorem 12.2, \cite{VinWang}]\label{bijalpha}
The canonical maps $P$ and $Q$ are bijections between $C^0(\tilde\alpha_0)$
and $C^0(\tilde\alpha)$, and these bijections send
$S(\tilde\alpha_0)$ to $S(\tilde\alpha)$.
\end{prop}

\begin{rmk}
The proof in \cite{VinWang} uses only the smallness of the $C^0$–distance
between the $\tilde\alpha$–foliations and the $\tilde\alpha_0$–foliations;
in our setting this is guaranteed by $\hat\alpha$ being a topological
perturbation of $\alpha_0$.
\end{rmk}

Projecting to $X=G/\Gamma$, this shows that the holonomy group of
$\hat\alpha$ coincides with the holonomy group of $\alpha_0$. Consequently
the holonomy group of $\hat\alpha$ acts freely and transitively on the universal cover of a center leaf.

\begin{prop}\label{hologroup}
The group of $su$–holonomies of $\hat\alpha$ acts freely and transitively
on $\tilde \W^c_{f_0}(x)$, for any $x\in X$.
\end{prop}

Returning to the Lyapunov cycle functional $P_\beta$, denote by
$C^{\tilde G}(\hat\alpha)$ the image of $C^0(\tilde\alpha)$ in
$C(\hat\alpha)$. By Lemma \ref{lem:stableinkernel} and Propositions \ref{minper} and \ref{bijalpha}, we have
$P_\beta(C^{\tilde G}(\hat\alpha))=\mathrm{id}$, so $P_\beta$ factors
through $C_{x_0}/C^{\tilde G}_{x_0}\simeq\tilde\Gamma$.

On the other hand we have the following lemma.

\begin{lem}[Lemma 11.1, \cite{VinWang}]\label{latticefin}
Let $\psi:\tilde\Gamma\to G^c$ be a sufficiently small
homomorphism. Then $\psi$ is trivial.
\end{lem}

Moreover, since $\beta$ is close to a constant, $P_\beta$ can be made
arbitrarily small.

\begin{lem}\label{pbetasmall}
For any $\epsilon>0$ there exists $\delta>0$ such that if
$d_{C^1}(f,f_0)<\delta$ and $\beta$ is a cocycle over $\hat\alpha$ that
is $\delta$–close to a constant cocycle $\beta_0$ on a generating set of
$\Z^2$, then $d(p_\beta(x,y),\mathrm{id})<\epsilon$ for any
$x\in X$, $y\in\mathcal W^i_\#(x)$ with $d_{\mathcal W^i_\#}(x,y)<1$.
\end{lem}

\begin{proof}
    Fix \( a \in \mathbb{Z}^2 \) such that \( \hat{\alpha}(a) \) is exponentially contracting on \( \W^i_\# \). Let \( g(x) = \hat{\alpha}(a, x) \). Note that the norm of \( a \) can be uniformly bounded. Define \( p_n = \beta(na, x)^{-1} \beta(na, y) \). Suppose \( g \) contracts \( \W^i_\# \) with exponent \( \chi < 0 \). Then by Proposition \ref{topph}, there exists a constant \( C \) such that
    \[
    d(p_{\beta}(x,y), \mathrm{id}) \le \sum_{i=0}^{\infty} d(p_i, p_{i+1}) \le \sum_{i=0}^{n} d(p_i, p_{i+1}) + \sum_{i=n+1}^{\infty} d(p_i, p_{i+1}).
    \]
    Now for any $\epsilon_0>0$, there exists $\delta>0$ such that,
    \[
    \sum_{i=0}^{n} d(p_i, p_{i+1}) \le n C \epsilon_0 \sup_{x\in X}d(\beta(a, x), \beta_0(a, x)),
    \]
    and
    \[
    \sum_{j=n+1}^{\infty} d(p_j, p_{j+1}) \le \sum_{j=n+1}^{\infty} C e^{\chi j} d(x,y)^{\theta^2}.
    \]
    By choosing \( n \) sufficiently large, the second term becomes small. Then by choosing \( \delta \) sufficiently small, the first term also becomes small, yielding the desired result.
\end{proof}

Therefore Theorem \ref{cocrig} follows from the propositions above.

\begin{proof}[Proof of Theorem \ref{cocrig}]
By Propositions \ref{minper} and \ref{bijalpha} the Lyapunov cycle
functional $P_\beta$ factors through $\tilde\Gamma$. By Lemma
\ref{pbetasmall}, $P_\beta$ can be made arbitrarily close to the
identity. Hence Lemma \ref{latticefin} implies that the image of
$P_\beta$ is trivial when $\beta$ is sufficiently close to a constant.
By Proposition \ref{consttrivialfunctional} the cocycle $\beta$ is
cohomologous to a constant.
\end{proof}

 \subsection{Elements of $\alpha$ are partially hyperbolic}\label{alphaph}

In this section we prove the following.

\begin{prop}\label{parthyp}
Let $f$ be a $C^1$-small perturbation of $f_0$ as in Theorems \ref{cenrigtwgen} or \ref{cenrignongen:twisted}. Let \(h\) be the Hölder conjugacy from Proposition~\ref{prop:fholconj},
so that \(h^{-1}\CZ_0(f)h\subset \CZ_0(f_0)\). Then any
$g\in\CZ_0(f)$ that corresponds to a generic element of the (twisted) Weyl
chamber flow is a partially hyperbolic diffeomorphism. In particular, this shows 
the images of the Lyapunov foliations of $g$ under the Hölder conjugacy
are smooth.
\end{prop}
The proof is very similar to the proof of Theorem 7.11 of \cite{Wang}. We will focus on the setting of Theorem \ref{cenrignongen:twisted}, since the case of a generic element of the twisted Weyl chamber flow is exactly the same as the proof of Theorem 7.11 of \cite{Wang}.

\begin{proof}
 For $g\in \CZ_0(f)$, let $g_0=h^{-1}gh$ be the corresponding generic element of the (twisted) Weyl chamber flow. We define the topological stable foliation of $g$ to be 
 $$\W^{s,\#}_g=h(\W^s_{g_0}).$$
 Similarly, we define $\W^{u,\#}_g=h(\W^u_{g_0})$ to be the topological unstable foliations of $g$.

 Let $\chi=\frac{\theta_1}{2} \cdot \chi(g_0)$ where $\theta_1<1$ is bi-H\"older exponent of $h$, and $\chi(g_0)<0$ is the contraction rate of $g_0$ in its stable bundle. Then by the H\"older conjugacy, we have for any $y\in \W^{s,\#}_g(x)$, \begin{equation}\label{equ:gcontract}
 d(g^n(x),g^n(y))\le e^{\chi n} d(x,y)^{\theta_1^2} 
 \end{equation} when $n\gg1$. 

 We take the center foliation of $g$ to be $$\W^c_g(x)=\{g_1(x):g_1\in \CZ_0(f), g_1\circ g=g\circ g_1\}$$ which is a smooth subfoliation of $\W^c_f$, since the action of $\CZ_0(f)$ on the center leaf is jointly smooth. Moreover, since \(h^{-1}gh=g_0\) and \(h^{-1}\CZ_0(f)h\subset \CZ_0(f_0)\), we have
\[
\W^c_g(x)
=
h\bigl(\W^c_{g_0}(h^{-1}(x))\bigr).
\]

Furthermore, $g$ acts ``almost isometrically" along its center. In fact, for any $y\in \W^c_g(x)$, then $y=g_1(x)$ for some $g_1\in \CZ_0(f)$ that commutes with $g$, so we have  $\min_{z\in X} d_{\W^c_f}(z,g_1(z))\le d_{\W^c_f}(g^n(x),g^n(y))=d_{\W^c_f}(g^n(x),g_1(g^n(x)))\le \max_{z\in X} d_{\W^c_f}(z,g_1(z)).$   Since the action by $g_1$ is free in the lifted center unless $g_1$ is the identity, we have for any $y\in \W^c_g(x)$, $y\neq x$, there exists $C>1$ such that for any $n\in \Z$, $$\frac{1}{C}<d_{\W^c_f}(g^n(x),g^n(y))<C.$$

The conjugacy also gives the following dynamical characterization of the topological stable, unstable foliation of $g$: 
 
 \begin{equation}\label{equ:dyncha}
 \begin{split}
 \W^{s,\#}_g(x)=\{y\in X: \limsup_{n\to \infty } \;d(g^n(x),g^n(y))=0\}\\
 \W^{u,\#}_g(x)=\{y\in X: \limsup_{n\to -\infty }\;d(g^n(x),g^n(y))=0\}
 \end{split}
 \end{equation}

Now we follow the same ideas as in \cite{Wang}. 
\textbf{Step 1:} We show that the foliation $\W^{s,\#}_g(x)$ agrees with the Pesin stable foliation at any $x\in X$ that is Pesin-regular for an ergodic measure $\nu$ of $g$. 

This is given by comparing Equations \ref{equ:gcontract} with the dynamical characterization of Pesin foliations and the bounded separation along \(\W^c_g\).

\textbf{Step 2:} We show the Lyapunov exponents along $\W^{s,\#}_g$ are $\le \chi$ for every $g$-ergodic measure.

This is given by the standard graph transformation argument in \cite[Theorem 3.16]{PuShergo}: if the Lyapunov exponent for some ergodic measure is $\beta > \chi$, then there exists a local disk embedded in the Pesin foliation, thus embedded in $\W^{s,\#}_g(x)$, that has dynamical growth with rate greater than $\chi$, which contradicts Equation \ref{equ:gcontract}.

\textbf{Step 3:} We use normal form theory to construct a continuous bundle $E^{i,s}_g$ which is contracted by $Dg$ inside the dominated splitting of $f$. 

We consider the dominated splitting of $f$, which has the same block size as the dominated splitting of $f_0$ (which is usually coarser than the splitting of the root spaces), $$TX=E^1_f\oplus E^2_f\oplus \cdots \oplus E^c_f\oplus \cdots \oplus E^t_f.$$ 

The non-central bundles $E^i_f, 1\le i\le t$ are continuous, and integrate to foliations $\W^i_f$ with smooth leaves. Moreover, $Df|_{E^i_f}$ satisfy the narrow band spectrum with critical regularity close to $1$. Therefore, by Theorem 2.4 of \cite{normalform} (see also Theorem 9 of \cite{DWX}), there exists normal-form coordinates $H_x:\W^i_f(x)\to E^i_f(x)$ such that $Q_x:=H_{g(x)}\circ g\circ H^{-1}_x$ is a linear map.

We denote $Q^{(n)}_x=H_{g^n(x)}\circ g^n\circ H^{-1}_x$, and define the subbundle $$E^{i,s}_{f,g}(x)=\{v\in E^i_f(x): \lim_{n\to \infty} \|Q^{(n)}_x(v)\|=0\}.$$ Then by the dynamical characterization of the foliations, we have $$H_x^{-1}(E^{i,s}_{f,g}(x))=\W^{s,\#}_g(x)\cap \W^i_f(x,loc).$$ 
Since $\W^{s,\#}_g$ has continuous leaves, $\W^{s,\#}_g(x)\cap \W^i_f(x,loc)$ has constant topological dimension, and $H_x$ depends $C^r$ continuously on $x$, this shows that $E^{i,s}_{f,g}(x)$ is a continuous $Dg$-invariant bundle. 

\textbf{Step 4:} By Step 2, the Lyapunov exponent of $Dg|_{E^{i,s}_{f,g}}$ is uniformly bounded by $\chi<0$ for any ergodic measure $\nu$ of $g$. Therefore, the classical estimates on subadditive sequences (see \cite{SCHREIBER1998334}, see also Lemma 10 of \cite{DWX}) show that for any $\epsilon \ll 1$, there exists $N > 1$ such that
$$\|Dg^n \mid_{E^{i,s}_{f,g}}\|\le e^{(\chi+\epsilon) n}$$
for any $n \ge N$ and $x \in X$. 

We also take stable foliations of $g$ inside the center foliation of $f$, given by $\W^{c,s}_{f,g}(x)=\{g_1(x): g_1\in \CZ_0(f), \limsup_{n\to \infty }d(g^ng_1g^{-n},id)=0\}$. By construction, this is a foliation with $C^\infty$-smooth leaves, and we have $\W^{c,s}_{f,g}=h(\W^c_{f_0}\cap \W^s_{g_0})$ by the conjugacy. Furthermore, by the group structure of $\CZ_0(f)\simeq \CZ_0(f_0)$, we have for any $y\in \W^{c,s}_{f,g}$, $d(g^n(x),g^n(y))\le e^{\chi n}$ for $n\gg 1$. Therefore, by the same argument, there exists a uniform $N>1$, such that $\|Dg^N|_{T\W^{c,s}_{f,g}(x)}\|\le e^{(\chi+\epsilon)N}$  for any $x\in X$

We take $E^s_g=T\W^{c,s}_{f,g}\oplus(\oplus_{i=1}^t E^{i,s}_{f,g})$. Then by our earlier steps $E^s_g$ is tangent to $\W^{s,\#}_g(x)$ at every $x\in X$, and $\dim E^s_g$ is equal to the topological dimension of leaves of $\W^{s,\#}_g$, and $Dg$ has uniform contraction along $E^s_g$, i.e., for any $n>N$ and $x\in X$, $$\|Dg^n|_{E^s_g(x)}\|\le e^{(\chi+\epsilon)n}.$$

The same argument can be applied to $g^{-1}$ to construct a continuous bundle $E^u_g$ and prove uniform expansion along $E^u_g$. 

In the center bundle $E^c_g=T\W^c_g$, by the ``almost isometric" property, the same argument as Step 2 shows that the Lyapunov exponent of $Dg|_{E^c_g}$ is zero for any $g$-invariant measure $\nu$. So, the classical estimates on subadditive sequences, applied to $Dg|_{E^c_g}$ and $Dg^{-1}|_{E^c_g}$ shows that for any $\epsilon>0$, there exists $N>1$, such that \[
\|Dg^n|_{E^c_g(x)}\|\le e^{\epsilon |n|},
\qquad
\|(Dg^n|_{E^c_g(x)})^{-1}\|\le e^{\epsilon |n|}
\] for any $x\in X$ and $|n|>N$.

Combining these with the fact that the bundles $E^s_g,E^c_g, E^u_g$ are $g$-invariant, and the dimensions add up to the ambient dimension, this shows that $g$ is partially hyperbolic.
\end{proof}

\bibliography{cenrig}{}
\bibliographystyle{acm}

\end{document}